\documentclass[11pt,reqno]{amsproc}

\title{Proof of modulational instability  of Stokes waves in deep water}

\author{Huy Q. Nguyen}
\address{Department of Mathematics, Brown University, Providence, RI 02912}
\email{hnguyen@math.brown.edu} 

\author{Walter A. Strauss}
\address{Department of Mathematics $\&$ Lefschetz Center for Dynamical Systems, Brown University, Providence, RI 02912}
\email{wstrauss@math.brown.edu}

\usepackage[margin=1in]{geometry}
\usepackage{amsmath, amsthm, amssymb, mathrsfs}
\usepackage{times}
\usepackage{color}
\usepackage{hyperref}

\usepackage{graphicx}
\usepackage{tikz}
\usetikzlibrary{decorations.pathmorphing}

\newcommand{\bq}{\begin{equation}}
\newcommand{\eq}{\end{equation}}
\newcommand{\bqa}{\begin{eqnarray*}}
\newcommand{\eqa}{\end{eqnarray*}}
\newcommand{\red}[1]{\textcolor{red}{#1}}
\newcommand{\blue}[1]{\textcolor{blue}{#1}}

\theoremstyle{plain}
\newtheorem{theo}{Theorem}[section]
\newtheorem{prop}[theo]{Proposition}
\newtheorem{lemm}[theo]{Lemma}
\newtheorem{coro}[theo]{Corollary}

\newtheorem{defi}[theo]{Definition}
\theoremstyle{definition}
\newtheorem{rema}[theo]{Remark} 
\newtheorem{nota}[theo]{Notation}
\DeclareMathOperator{\cnx}{div}
\DeclareMathOperator{\RE}{Re}
\DeclareMathOperator{\dist}{dist}
\DeclareMathOperator{\IM}{Im}
\DeclareMathOperator{\Op}{Op}
\DeclareMathOperator{\curl}{curl}
\DeclareMathOperator{\supp}{supp}
\DeclareMathOperator{\di}{d}
\DeclareSymbolFont{pletters}{OT1}{cmr}{m}{sl}
\DeclareMathSymbol{s}{\mathalpha}{pletters}{`s}

\def\sign{\mathrm{sign}}
\def\mb{\mathfrak{b}}\def\tt{\theta}
\def\eps{\varepsilon}
\def\na{\nabla}
\def\la{\left\lvert}
\def\lA{\left\lVert}
\def\ra{\right\vert}
\def\rA{\right\Vert} 
\def\les{\lesssim}
\def\lg{\langle}
\def\rg{\rangle}
\def\wh{\widehat}
\def\mez{\frac{1}{2}}
\def\tdm{\frac{3}{2}}
\def\tq{\frac{3}{4}}
\def\uq{\frac{1}{4}}

\def\Rr{\mathbb{R}}
\def\Nn{\mathbb{N}}
\def\Zz{\mathbb{Z}}
\def\Cc{\mathbb{C}}
\def\Tt{\mathbb{T}}
\def\P{\mathbb{P}}
\def\RHS{\text{RHS}}
\def\cE{\mathcal{E}}
\def\cF{\mathcal{F}}
\def\cK{\mathcal{K}}
\def\cN{\mathcal{N}}
\def\cP{\mathcal{P}}
\def\cT{\mathcal{T}}
\def\cL{\mathcal{L}}
\def\cU{\mathcal{U}}
\def\cO{\mathcal{O}}
\def\ld{\lambda}
\def\Ld{\lambda}
\def\p{\partial}
\def\na{\nabla}
\def\wc{\rightharpoonup}
\def\ka{\kappa}
\def\ma{\mathfrak{a}}
\def\mc{\mathfrak{c}}
\def\ol{\overline}
\def\T{\mathbb{T}}
\def\wt{\widetilde}
\def\cM{\mathcal{M}}
\def\ka{\kappa}
\newcommand{\hk}{\hspace*{.15in}}
\numberwithin{equation}{section}
\def\om{\omega}
\def\ul{\underline}
\def\red {\color{red} } 
\def\blue {\color{blue} } 

\date{\today}
\begin{document}
\begin{abstract}
It is proven that small-amplitude steady periodic water waves with infinite depth are
unstable with respect to long-wave perturbations.  This modulational instability
was first observed  more than half a century ago by Benjamin and Feir. It has  been proven rigorously only in the case of finite depth. We provide a completely different and self-contained approach to prove the spectral modulational instability for water waves in both the finite and infinite depth cases. 
\end{abstract}



\maketitle

\section{Introduction}
We consider classical water waves in two dimensions that are irrotational,  inviscid and horizontally periodic.  
The water is below a free surface $S$ and has infinite depth.  
 Such waves have  been studied for over two centuries, notably by Stokes \cite{Stokes}.  
A Stokes wave is a steady wave traveling at a fixed speed $c$.  It has been known for a century that 
a curve of small-amplitude Stokes waves exists \cite{Nekrasov, Civita, Struik}.  
In 1967 Benjamin and Feir \cite{BenjFeir} discovered that a small  long-wave  perturbation 
of a small Stokes wave will lead to exponential instability.  
This is called the {\it modulational} (or Benjamin-Feir or sideband) {\it instability}, a phenomenon 
 whereby deviations from a periodic wave are reinforced by the nonlinearity, leading to the eventual breakup of the wave into a train of pulses.  Here we provide a complete proof of this instability for deep water waves.  

To be a bit more specific, let $x$ be the horizontal variable and $y$ the vertical one.  
Consider the curve of steady waves  of a given period,  say $2\pi$ without loss of generality, to be parametrized by a small parameter $\eps$ which represents the wave  amplitude.  
Such a steady wave can be described in the moving plane (where $x-ct$ is replaced by $x$) 
by its free surface $S=\{y=\eta^*(x;\eps)\}$ and its velocity potential $\psi^*(x;\eps)$ restricted to $S$.  
We use a conformal mapping of the fluid domain to the lower half-plane, thereby converting the whole 
problem to a problem with a fixed flat surface.  Let the perturbation have  a small wavenumber $\mu$;  that is, we have introduced a long wave.  Linearization around the steady wave leads to a linear operator $\cL_{\mu,\eps}$.  
What we prove is the spectral instability, which means that the perturbed water wave grows in time 
like $e^{\ld t}$ for some complex number $\ld$ with positive real part.   
A way to state this formally is as follows.   
\begin{theo}\label{maintheorem}
There exists $\eps_0>0$ such that for all $0<|\eps|<\eps_0$, there exists $\mu_0=\mu_0(\eps)>0$ such that for all $0<|\mu|<\mu_0$, the 
 operator $\cL_{\mu,\eps}$ has an eigenvalue $\ld$ with positive real part.  Moreover, $\ld$ has the asymptotic expansion 
\bq \label{eigenvalue} 
\ld=\frac{\sqrt{g}}{2} i\mu +  \frac{\sqrt{g}}{2\sqrt2}|\mu\eps| + O(\mu^2) + O(\mu\eps^2),
\eq
where $g>0$ is the  acceleration due to gravity.
\end{theo}
%
The concept of modulational instability arose in multiple contexts in the 1960's, both in the 
theory of fluids including water waves and in electromagnetic theory including laser beams and 
plasma waves.  MathSciNet lists more than 500 papers mentioning  ``modulational instability" or ``Benjamin-Feir instability".  
Major players in its early history included Lighthill 1965, Whitham 1967, Benjamin 1967 and Zakharov 1968,   
as described historically in \cite{ZakOst}. 
It was a surprising development when Benjamin and Feir \cite{BenjFeir, Benj} discovered the phenomenon 
in the context of the full theory of water waves, as they did 
both theoretically and experimentally (see also \cite{Whitham:book, Whitham}).  They identified the  
most dominant plane waves that can arise from small disturbances of the steady wave. 
However, to make a completely rigorous proof of the instability is another matter.   This is our focus.
It took about three decades for such a proof to be found for the case of finite depth.  
Bridges and Mielke \cite{BM} accomplished the feat by means of a spatial dynamical reduction to a four-dimensional center manifold.   Nevertheless, their proof {\it cannot} be generalized to the case of infinite depth 
due to the lack of compactness, which invalidates the hypotheses of the center manifold theory. The infinite depth case  has remained unsolved since then. After the completion \cite{NguyenStrauss} of the current paper, we learned of another proof \cite{HurYang} of the spectral instability which also does not generalize to infinite depth. 
In the current paper we provide a  {\it completely different approach}
to prove the modulational instability of small-amplitude Stokes waves.  Our  proof is 
self-contained, does not rely on any abstract Hamiltonian theory, and encompasses {\it both the finite and infinite depth cases}.  
In order to avoid tedious algebra, we focus on the unsolved case of infinite depth and shall merely point out 
the main modifications necessary for the finite depth case.  
As distinguished from \cite{BM}, throughout our proof the physical variables are retained.  
Our linearized system is obtained from the Zakharov-Craig-Sulem formulation  together with the 
use of Alinhac's ``good unknown" and with a Riemann mapping.  Thus it is compatible with the Sobolev energy 
estimates for the nonlinear system (see e.g \cite{Lannes, AlaMet, ABZ1, ABZ3, RouTzv}). After the completion \cite{NguyenStrauss} of the current paper, we learned of the paper \cite{ChenSu} by Chen and Su which uses an approximation to the focusing cubic nonlinear Schr\"odinger equation (NLS)  to indirectly deduce the nonlinear instability. On the other hand, we expect that the framework developed in our paper should be useful to directly prove the nonlinear instability without any reference to NLS.

There have been many studies of the modulational instability for a variety of approximate water wave models, 
 such as KdV, NLS and the Whitham equation
by, for instance, Whitham \cite{Whitham:book},  Segur,  Henderson,  Carter and  Hammack \cite{SegHenCarHam}, Gallay and Haragus \cite{GalHar},  Haragus and Kapitula \cite{HarKap}, 
Bronski and Johnson \cite{BronJohn}, Johnson \cite{Johnson}, Hur and Johnson \cite{HurJohn} and  Hur and Pandey \cite{HurPan}.  
These models are surveyed in \cite{BronJohnHur}.  Beyond the linear modulational theory, a proof of the nonlinear modulational instability 
for several of the models is given in \cite{JinLiaLin}. That is, an appropriate Sobolev norm of a long-wave perturbation to the nonlinear problem grows in time. There have also been many numerical studies on this phenomenon.  We mention the paper  by 
Deconinck and Oliveras \cite{DecOli}, which provides a detailed description of the 
unstable solutions including pictures of the unstable manifold of solutions far from the bifurcation, 
a rigorous proof of which remains largely open. On the other hand, the asymptotic expansion 
\eqref{eigenvalue} does show that the unstable eigenvalue, as a curve  with parameter $\mu$, 
has  slope  $\sim |\eps|^{-1}{\sign(\mu)\sqrt{2}}$ near the origin in the complex plane.  This agrees well with the numerical calculation shown in the following figure \cite{Creedon}.
 \begin{figure}[hpt!]
  \centering
  \includegraphics[width=.5\linewidth]{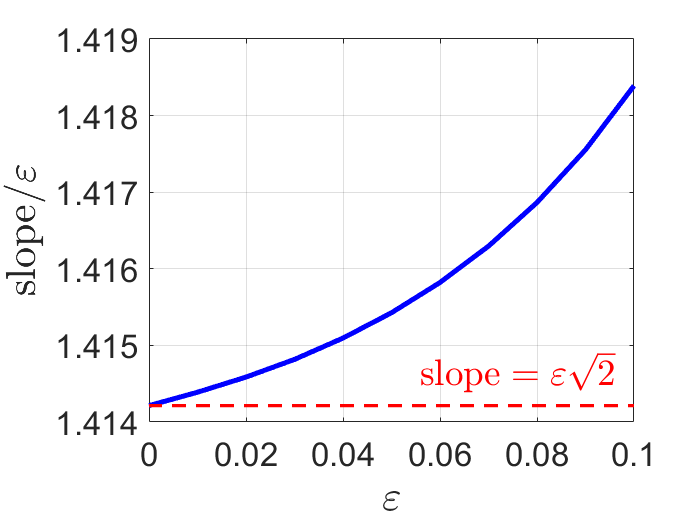}
\end{figure}


Now we outline the contents of this paper.  

In Section 2 we write the water wave equations in the Zakharov-Craig-Sulem formulation.  
Thus the system is written in terms of the pair of functions $\eta$, which describes the free surface $S$ and 
$\psi$, which is the velocity potential on $S$.  This formulation involves the Dirichlet-Neumann 
operator $G(\eta)$, which is non-local.  The advantages of this formulation are that  $\eta$ and $\psi$ 
depend on only the single variable $x$ and that the system has Hamiltonian form.  
Stokes' steady wave $\eta^*(x;\eps),\psi^*(x;\eps)$ is then expanded in powers of $\eps$ up to $\eps^3$.   
Such an expansion basically goes back to Stokes himself, although the literature can be confusing so 
we include a proof in Appendix A.  We note however that the proof of our main result only requires expansions up to $\eps^2$.

Section 3 is devoted to the linearization, using the shape-derivative formula of \cite{Lannes} and Alinhac's good unknown.  
Then we flatten the boundary by using the conformal mapping between the fluid domain and the lower 
half-plane.  This converts the implicit nonlocal operator $G(\eta)$ to the explicit Fourier multiplier $G(0)=|D|$. A direct proof is given in Appendix B.   We look for solutions of the form $e^{i\mu x}U(x;\eps)$, where  $U(\cdot,\eps)$ has period $2\pi$ and  a small $\mu$ represents a long-wave perturbation.  
The unknowns are $U=$ the pair $(\eta, \text{good unknown})$, appropriately modified by the conformal mapping.   This brings us to the linearized operator $\cL_{\mu,\eps}$, which acts 
from $(H^1(\T))^2$ to $(L^2(\T))^2$.  It is Hamiltonian.  
The instability problem is thereby reduced to finding an eigenvalue $\ld(\mu,\eps)$ of $\cL_{\mu,\eps}$ 
with positive real part.

We put $\mu=0$ in Section 4.  It is shown that $\cL_{0,\eps}$ has a two-dimensional nullspace 
and a four-dimensional generalized nullspace $\cU(\eps)$.  Then we construct an {\it explicit} basis of $\cU(\eps)$, denoted by 
$\{U_1(\eps),\dots U_4(\eps)\}$.  This construction works for both the finite and infinite depth cases and is the starting point of our proof. We expand each $U_j(\eps)$ in powers of $\eps$.  Then we compute the nullspace and range of the operator $\Pi \cL_{0,\eps}$ where $\Pi$ is the projection onto  the orthogonal complement of $\cU(\eps)$. This will be crucially used in searching for a bifurcation from $\cU(\eps)$ when $\mu$ is nonzero. 

Now with $\mu\ne0$ in Section 5 we expand the inner products $(\cL_{\mu,\eps} U_j,U_k)$ in powers of both 
parameters $\mu$ and $\eps$.  Our procedure of looking at the inner products roughly follows 
the procedure of Johnson \cite{Johnson} and Hur and Johnson \cite{HurJohn}, who carried it out in their stability analysis for KdV-type equations and the Whitham equation,  which followed several earlier works cited above. 

Of course, for fixed $\eps$ the perturbation due to $\mu\ne0$ will change the vanishing eigenvalue 
to $\ld(\mu,\eps)\ne0$.   The associated eigenfunction will have a small component outside of $\cU(\eps)$; 
that is, it will have the form $\sum_{j=1}^4 \alpha_j (U_j(\eps) + W_j(\mu,\eps))$.  
We call $W_j$ the {\it sideband functions}. Perturbation theory for linear operators merely asserts that 
each $W_j(\mu,\eps)$ is small if $\mu$ is small enough (see \cite{Kato}).  
In Subsection 6.1 we treat these sideband functions  by means of a rather subtle version of the Lyapunov-Schmidt method 
that uses the inverse of the operator $\Pi \cL_{0,\eps}$  obtained in Section 4. 
In Subsection 6.2 we expand $(\cL_{\mu,\eps} W_j,U_k)$ in powers of $(\mu,\eps)$ up to second order in $\eps$.  

In Section 7 we combine the asymptotic expansions of Sections 5 and 6.  The key task is to identify the leading terms 
and to handle the numerous remainder terms.  Surprisingly, it turns out that one of the key leading terms 
comes from $(\cL_{\mu,\eps} W_j,U_k)$, namely the one that we denote by $II_{10}$ in \eqref{expand:II}.  
 That is, it is the combination of the expansions of $(\cL_{\mu,\eps} U_j,U_k)$ and $(\cL_{\mu,\eps} W_j,U_k)$ 
that lead to the required result. We remark that in the works cited above, the sideband functions were always treated as negligible remainders; it is different for this full water wave problem.   Finally we use the expansions to deduce that there is an eigenvalue of the form \eqref{eigenvalue},  
which obviously has a positive real part.

 The explicit expansions require detailed calculations.  We have carried them out all the way to third order, 
which is more than necessary for our instability proof, but has potential utility in future  
theoretical and numerical research.  We have summarized these expansions in Appendix D.  


\section{The Zakharov-Craig-Sulem formulation and Stokes waves}
We consider the fluid domain 
\bq
\Omega(t)=\{(x, y): x\in \Rr,~y<\eta(x, t)\}.
\eq
 below the  free surface $S=\{(x , \eta(x, t)): x\in \Rr\}$ to have infinite depth.
Assuming that the fluid is incompressible,  inviscid  and irrotational, the velocity field admits a harmonic 
potential $\phi(x, y, t):\Omega\to \Rr$. Then $\phi$ and $\eta$ satisfy the water wave system
\bq\label{ww:0}
\begin{cases}
 \Delta_{x, y} \phi =0\quad \text{ in } \Omega, \\
\p_t\phi + \tfrac12 |\na_{x, y}\phi|^2 =- g\eta +P\quad \text{ on } \{y=\eta(x)\}, \\
 \p_t\eta+\p_x\phi\p_x\eta=\p_y\phi\quad  \text{ on } \{y=\eta(x)\}, \\
 \na_{x, y}\phi \to 0 \text{ as } y\to -\infty,
\end{cases}
\eq
where $P\in \Rr$ denotes the Bernoulli constant and $g>0$ is the constant acceleration due to gravity. The second equation is Bernoulli's,  which follows from the 
pressure being constant along the free surface; the third equation expresses the kinematic boundary condition that particles on the surface remain there; 
the last condition asserts that the water is quiescent at great depths.

 In order to reduce the system to the free surface $S$, we introduce the Dirichlet-Neumann operator $G(\eta)$ 
 associated to $\Omega$, namely, 
\bq\label{def:Gh}
G(\eta)f=\p_y\tt(x, \eta(x))-\p_x\tt(x, \eta(x))\p_x\eta(x),
\eq
where $\tt(x, y)$ solves the elliptic problem 
\bq\label{elliptic:G}
\begin{cases}
\Delta_{x, y}\tt=0\quad\text{in}~\Omega,\\
\tt\vert_{y=\eta(x)}=f(x),\quad \na_{x, y} \tt\in L^2(\Omega).
\end{cases}
\eq
Let $\psi$ denote the trace of the velocity potential on the free surface, $\psi(t, x)=\phi(t, x, \eta(t, x))$.   
In the moving frame with speed $c$, the gravity water wave system written in the Zakharov-Craig-Sulem 
formulation \cite{Zak, CraSul} is 
 \bq\label{ww:c}
 \begin{cases}
 \p_t \eta=c\p_x\eta+ G(\eta)\psi,\\
 \p_t\psi=c\p_x\psi-\mez|\p_x\psi|^2+\mez\frac{\big( G(\eta)\psi+\p_x\psi\p_x\eta\big)^2}{1+|\p_x\eta|^2}-g\eta+P.
 \end{cases}
 \eq
 By a steady wave we mean that $\eta$ is a function of $x-ct$ and $\phi$ a function of $(x-ct,y)$. 
 By a {\it Stokes wave} we mean a periodic steady solution of \eqref{ww:c}; that is, 
  \bq\label{sys:Stokes}
 \begin{cases}
F_1(\eta, \psi, c):=c\p_x\eta+ G(\eta)\psi=0,\\
 F_2(\eta, \psi, c, P):=c\p_x\psi-\mez|\p_x\psi|^2+\mez\frac{\big( G(\eta)\psi+\p_x\psi\p_x\eta\big)^2}{1+|\p_x\eta|^2}-g\eta+P=0.
 \end{cases}
 \eq
   \begin{figure}[hpt!]
  \centering
  \includegraphics[width=.7\linewidth]{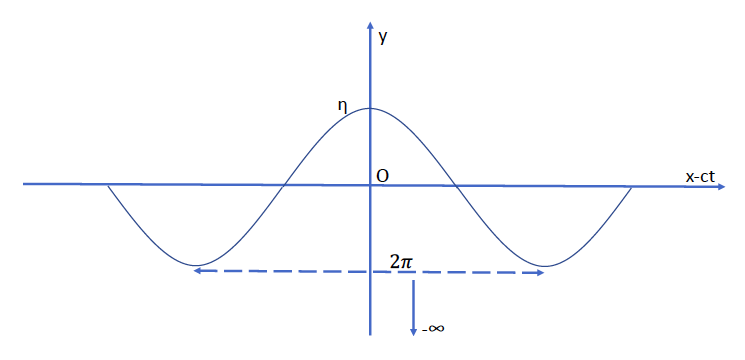}
  \caption{Stokes wave}
\end{figure}

 The existence of a smooth local curve of smooth steady solutions satisfying ($i$) and ($ii$) below has been known for 
 a century, going back to Nekrasov \cite{Nekrasov} and Levi-Civita \cite{Civita}. 
 \begin{theo}\label{theo:Stokes}
 For all  $P\in \Rr$, there exists a curve of smooth steady solutions $(\eta, \psi, c, P)$  to  \eqref{sys:Stokes} 
 parametrized by the amplitude $|a|\ll 1$ and the Bernoulli constant $P\in \Rr$ such that 
 
 (i)  $\eta$ and $\psi$ are $2\pi$-periodic,
 
 (ii) $\eta$ is even and $\psi$ is odd. 
 
 Other than the trivial solutions (with $\eta\equiv0$), the curve is unique. These solutions are called Stokes waves.
  \end{theo} 
  It is readily seen that system \eqref{sys:Stokes} respects the evenness of $\eta$ and the oddness of $\psi$.  
  Expansions of Stokes waves with respect to the amplitude $a$ are given in the next proposition.
  \begin{prop}\label{steadyexpansion}
 The following  expansions hold for the solutions in Theorem \ref{theo:Stokes}.   
\bq\label{expand:Stokes}
\begin{aligned}
&\eta=\frac{P}{g}+a \cos x+\mez a^2\cos(2x)+a^3\Big\{\frac{1}{8}\cos x+\frac{3}{8}\cos (3x)\Big\}+O(a^4),\\
& \psi= a \sqrt{g}\sin x+\frac{\sqrt{g}}{2}a^2\sin (2 x)+\frac{\sqrt{g}}{4} a^3\big\{3\sin x\cos(2x)+\sin x\big\}+O(a^4),\\
&c=\sqrt{g}+\frac{\sqrt{g}}{2}a^2+O(a^3).
\end{aligned}
\eq
 \end{prop}
 \begin{proof}
Proposition  \ref{steadyexpansion} essentially goes back to Stokes \cite{Stokes}. 
For the sake of precision and completeness, we give a detailed derivation in Appendix \ref{appendix:Stokes} 
for zero Bernoulli constant, $P=0$.  
Consider now the case $P\ne 0$. Setting $\eta=\wt \eta+\frac{P}{g}$ and using the facts that 
\[
G\big(\wt \eta+\frac{P}{g}\big)\psi=G(\wt \eta)\psi,\quad -g\eta+P=-g\wt \eta,
\]
we obtain 
\[
 \begin{cases}
F_1(\eta, \psi, c)=c\p_x\wt \eta+ G(\wt \eta)\psi,\\
 F_2(\eta, \psi, c, P)=c\p_x\psi-\mez|\p_x\psi|^2  
 +\mez\frac{\big(G(\wt \eta)\psi+\p_x\psi\p_x\wt\eta\big)^2}{1+|\p_x\wt\eta|^2}-g\wt \eta,
 \end{cases}
 \]
 thereby reducing us to the case $P=0$. 
 \end{proof}
 \section{Linearization and Riemann mapping} 
 We begin with  notation for $L$-periodic functions.  Set 
 \[
 \T_L=\Rr/L\Zz,\quad \T\equiv T_{2\pi}.
 \]
Let $f:\Rr\to \Rr$ be $L$-periodic. The $L$-Fourier coefficient of $f$ is 
\bq\label{Fourier:coefficient}
\wh{f}^L(k)=\int_0^L e^{-i\frac{2\pi}{L}kx}f(x)dx\quad\forall k\in \Zz,\quad \wh{f}(k)\equiv\wh{f}^{2\pi}(k).
\eq
For $m:\Rr\to \Rr$, the Fourier multiplier $m(D_L)$ is defined by 
\bq\label{Fourier:multiplier}
m(D_L)f(x)=\frac{1}{L}\sum_{k\in \Zz} e^{i\frac{2\pi}{L}kx}m\big(k\frac{2\pi}{L}\big)\wh{f}^L(k),\quad m(D)f(x)\equiv m(D_{2\pi})f(x).
\eq
 \subsection{Linearization} 
Fix $(\eta^*, \psi^*, c^*, P^*=0)$ a  solution of \eqref{sys:Stokes} as given in Theorem \ref{theo:Stokes} with  $a=\eps$, $|\eps|\ll 1$.   The expansions in \eqref{expand:Stokes} give 
\bq\label{expand:star}
\begin{aligned}
&\eta^*=\eps \cos x+\mez \eps^2\cos(2x)+\eps^3\{\frac{1}{8}\cos x+\frac{3}{8}\cos (3x)\}+O(\eps^4),\\
& \psi^*= \eps \sqrt{g}\sin x+\frac{\sqrt{g}}{2}\eps^2\sin (2 x)+\frac{\sqrt{g}}{4} \eps^3\{3\sin x\cos(2x)+\sin x\}+O(\eps^4),\\
&c^*=\sqrt{g}+\frac{\sqrt{g}}{2} \eps^2+O(\eps^3).
\end{aligned}
\eq
 We investigate the modulational instability of $(\eta^*, \psi^*, c^*, P^*)$ subject to {\it perturbations in $\eta$ and $\psi$ but not in $c^*$ and $P^*$}.  We shall consider {\it $L$-periodic perturbations  of $\eta$ and $\psi$, where {\it $L=n_02\pi$ for some integer $n_0$}.} In order to linearize \eqref{sys:Stokes} with respect to the free surface $S$, we make use of the so-called ``shape-derivative".  
The following statement and its proof are found in \cite{Lannes}.  
 \begin{prop}  
 For $L$-periodic functions, the derivative of the map $\eta\mapsto G(\eta)\psi$ is given by
\bq\label{shapederi}
 \frac{\delta G(\eta)\psi}{\delta \eta}(\ol\eta) =-G(\eta)(B\ol \eta)-\p_x(V\ol\eta),
\eq
where 
\bq\label{BV}
B=B(\eta, \psi)=\frac{G(\eta)\psi+\p_x\psi\p_x\eta}{1+|\p_x\eta|^2},\qquad V=V(\eta, \psi)=\p_x\psi-B\p_x\eta.
\eq
In fact, $V=(\p_x\tt)(x, \eta(x))$ and $B=(\p_y\tt)(x, \eta)$, where $\tt$ solves \eqref{elliptic:G}.  Moreover, if $\eta$ is even and $\psi$ is odd, then $B$ is odd and $V$ is even.
\end{prop}

\begin{lemm}
 We have
\begin{align}\label{dF1}
&\frac{\delta F_1(\eta, \psi, c)}{\delta(\eta, \psi)}(\ol\eta, \ol \psi)=\p_x\big((c-V)\ol \eta\big)+G(\eta)\big(\ol \psi-B\ol\eta\big),\\ \label{dF2}
&\frac{\delta F_2(\eta, \psi, c, P)}{\delta(\eta, \psi)}(\ol\eta, \ol \psi)=(c-V)\p_x\ol \psi+{B}G(\eta)(\ol \psi-{B}\ol \eta)-{B}\p_x{V}\ol\eta-g\ol \eta
\end{align}
together with the identity 
\bq\label{identity}
\frac{\delta F_2(\eta, \psi, c, P)}{\delta(\eta, \psi)}(\ol\eta, \ol \psi)-B\frac{\delta F_1(\eta, \psi, c)}{\delta(\eta, \psi)}(\ol\eta, \ol \psi)=-\big(g+({V}-c)\p_x{B} \big)\ol\eta+(c-{V})\p_x(\ol \psi-B\ol \eta),
\eq
where $B=B(\eta, \psi)$ and $V=V(\eta, \psi)$ are given by \eqref{BV}.  
\end{lemm}
\begin{proof}  
We note that \eqref{dF1} is a direct consequence of \eqref{shapederi}. As for $F_2$ we first compute  
\begin{align*}
&\frac{\delta}{\delta(\eta, \psi)}\frac{\big(G(\eta)\psi+\p_x\psi\p_x\eta\big)^2}{1+|\p_x\eta|^2}(\ol\eta, \ol \psi)\\
&=2\frac{\big(G(\eta)\psi+\p_x\psi\p_x\eta\big)}{1+|\p_x\eta|^2}\Big[G(\eta)\ol \psi-G(\eta)({B}\ol \eta)-\p_x({V}\ol\eta)+\p_x\psi\p_x\ol\eta+\p_x\ol\psi\p_x\eta\Big]
\\
&\quad -\big(G(\eta)\psi+\p_x\psi\p_x\eta\big)^2\frac{2\p_x\eta\p_x\ol \eta}{(1+|\p_x\eta|^2)^2}\\
&=2{B}\Big[G(\eta)\ol \psi-G(\eta)({B}\ol \eta)-\p_x({V}\ol\eta)+\p_x\psi\p_x\ol\eta+\p_x\ol\psi\p_x\eta\Big]-2{B}^2\p_x\eta\p_x\ol \eta.
\end{align*}
Consequently
\[
\begin{aligned}
&\frac{\delta F_2(\eta, \psi, c, P)}{\delta(\eta, \psi)}(\ol\eta, \ol \psi)\\
&=c\p_x\ol \psi-\p_x\psi\p_x\ol \psi+{B}G(\eta)\ol \psi-{B}G(\eta)({B}\ol \eta)-{B}\p_x{V}\ol\eta -{B}{V}\p_x\ol \eta\\
&\qquad +{B}\p_x\psi\p_x\ol \eta+{B}\p_x\ol\psi\p_x\eta-{B}^2\p_x\eta\p_x\ol \eta-g\ol \eta\\
&=c\p_x\ol \psi -\p_x\ol\psi\big(\p_x\psi-{B}\p_x\eta)+{B}G(\eta)\ol \psi-{B}G(\eta)({B}\ol \eta)-{B}\p_x{V}\ol\eta +{B}(\p_x\psi-{V})\p_x\ol \eta\\
&\quad-{B}^2\p_x\eta\p_x\ol \eta-g\ol \eta\\
&=c\p_x\ol \psi-\p_x\ol\psi {V}+{B}G(\eta)\ol \psi-{B}G(\eta)({B}\ol \eta)-{B}\p_x{V}\ol\eta+{B}^2\p_x\eta\p_x\ol \eta-{B}^2\p_x\eta\p_x\ol \eta-g\ol \eta\\
&=(c-V)\p_x\ol \psi+{B}G(\eta)(\ol \psi-{B}\ol \eta)-{B}\p_x{V}\ol\eta-g\ol \eta
\end{aligned}
\]
which proves \eqref{dF2}. Finally, a combination of \eqref{dF1} and \eqref{dF2} gives \eqref{identity}.
\end{proof}
 From \eqref{dF1} and \eqref{dF2} we obtain  the linearized system for \eqref{ww:c} about $(\eta^*, \psi^*, c^*, P^*)$ with $(c^*, P^*)$ being fixed:
\begin{align}
&\p_t \ol\eta=\frac{\delta F_1(\eta^*, \psi^*, c^*)}{\delta(\eta, \psi)}(\ol\eta, \ol \psi)=\p_x\big((c^*-V^*)\ol\eta\big)+G(\eta^*) (\ol\psi-{B^*} \ol\eta),\label{lin:eta}\\
&\p_t\ol \psi=\frac{\delta F_2(\eta^*, \psi^*, c^*, P^*)}{\delta(\eta, \psi)}(\ol\eta, \ol \psi)=(c^*-{V^*})\p_x\ol\psi +{B^*}G(\eta^*)(\ol\psi-{B^*} \ol\eta)-{B^*}\p_x{V^*}\ol\eta-g\ol\eta,\label{lin:psi}
\end{align}
where $B^*$ and $V^*$ are given in terms of $\eta^*$ and $\psi^*$ as in \eqref{BV},  and $\ol \eta$ and $\ol \psi$ are $L$-periodic. By virtue of identity \eqref{identity}, the   {\it good unknowns} (\`a la Alinhac \cite{Alinhac, AlaMet})
\bq\label{def:v12}
v_1=\ol\eta, \quad v_2=\ol\psi-{B^*}\ol\eta,
\eq
 satisfy
 \begin{align}\label{eq:v1}
&\p_t v_1=\p_x\big((c^*-{V^*})v_1\big)+G(\eta^*)v_2,\\ \label{eq:v2}
&\p_t v_2=-\big(g+({V^*}-c^*)\p_x{B^*} \big)v_1+(c^*-{V^*})\p_xv_2.
\end{align}
The good unknowns \eqref{def:v12} have been successfully used in well-posedness and stability results 
for the nonlinear water wave system in spaces of finite regularity. See \cite{Lannes, AlaMet, ABZ1, ABZ3, RouTzv}.
\subsection{Conformal mapping} 
Due to the nontrivial  surface $\eta^*$, the Dirichlet-Neumann operator $G(\eta^*)$ appearing in the 
linearized system \eqref{eq:v1}-\eqref{eq:v2} is not explicit.  
 Analogously to \cite{PegoSun},  we  use the  Riemann mapping in the following proposition to flatten 
the free surface $S=\{(x, \eta^*(x)): x\in \Rr\}$.
 \begin{prop}\label{prop:z12}
 There exists a  holomorphic bijection $z(x, y)=z_1(x, y)+iz_2(x, y)$ from $\Rr^2_-=\{(x, y)\in \Rr^2: y<0\}$  onto $\{(x, y)\in \Rr^2: y<\eta^*(x)\}$ with the following properties.  
 \begin{itemize}
 \item[(i)] $z_1(x+2\pi, y)=2\pi +z_1(x, y)$ and $z_2(x+2\pi, y)=z_2(x, y)$ for all $(x, y)\in \Rr^2_-$; 
 \newline $z_1$ is odd in $x$ and $z_2$ is even in $x$;
 \item[(ii)] $z$ maps  $\{(x, 0): x\in \Rr\}$ onto $\{(x, \eta^*(x)): x\in \Rr\}$;\\
 \item[(iii)] Defining the ``Riemann stretch" as 
 \bq\label{def:zeta}
 \zeta(x)=z_1(x, 0),
 \eq 
we have the Fourier expansion 
 \bq\label{z1}
z_1(x, y)=x-\frac{i}{2\pi}\sum_{k\ne 0}e^{ikx}\mathrm{sign}(k)e^{|k|y}\widehat{\eta^*\circ \zeta}(k)\quad\forall (x, y)\in \Rr^2_-,
\eq
where 
\[
\wh{f}(\xi)=\int_{\Rr}e^{-i\xi x}f(x)dx.
\]
\item[(iv)] $\| \na_{x, y}(z_1-x)\|_{ L^\infty(\Rr^2_-)}+ \|\na_{x, y}(z_2-y)\|_{ L^\infty(\Rr^2_-)}\le C\eps$.
 \end{itemize}
 \end{prop}
    \begin{figure}
  \centering
  \includegraphics[width=.5\linewidth]{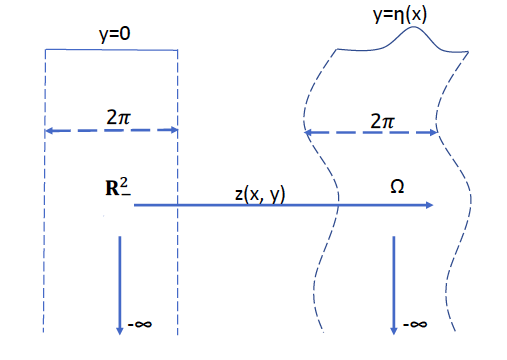}
  \caption{The Riemann mapping $z=z_1+iz_2$}
\end{figure}
 We postpone the proof of Proposition \ref{prop:z12} to Appendix \ref{Appendix:Riemann}. Compared to the finite depth case in \cite{PegoSun}, the proof of Proposition \ref{prop:z12} requires  decay properties as $y\to -\infty$. 
 
 In terms of the Riemann stretch $\zeta$, we can rewrite the Dirichlet-Neumann operator $G(\eta^*)$ as follows.  Define two operators 
 \bq
\zeta_\sharp f=f\circ \zeta ,  \qquad   \zeta_*f=\zeta' (f\circ \zeta), 
\eq  
so that  $\zeta_*\p_xf=\p_x(\zeta_\sharp f)$ for $f:\Rr\to \Rr$. 
\begin{lemm}\label{lemm:Riemann}
 For $f\in  H^1(\T_L)$ we have
\bq\label{DN:zeta}
G(\eta^*)f=\p_x\big(\zeta^{-1}_\sharp \mathcal{H}_L(\zeta_\sharp f)\big)
\eq
where  $\wh{\mathcal{H}_Lu}^L(k)=-i\mathrm{sign}(k)\wh{u}^L(\xi)$ is the Hilbert transform. The sign function $\sign:\Rr\to \{-1, 0, 1\}$ is defined as  
\bq
\sign(x)=
\begin{cases}
1\quad\text{if}~x>0, \\
0\quad\text{if}~x=0,\\
-1\quad\text{if}~x<0.
\end{cases}
\eq
 \end{lemm}
The proof of Lemma \ref{lemm:Riemann} is also given in Appendix \ref{Appendix:Riemann}.   

By virtue of Lemma \ref{lemm:Riemann},  for any functions $f_1,\, f_2 \in H^1(\T_L)$ a direct calculation yields the identities 
\bq\label{identity:zeta}
\begin{aligned}
&\zeta_*\Big(\p_x\big((c^*-{V^*})f_1\big)+G(\eta^*)f_2\Big)=\p_x\big(p(x)\zeta_*f_1\big)+|D_L|(\zeta_\sharp f_2),\\
&\zeta_\sharp\Big(-\big(g+({V^*}-c^*)\p_x{B^*} \big)f_1+(c^*-{V^*})\p_xf_2\Big)  
=  -\frac{g +q(x)}{\zeta'(x)}\zeta_*f_1+p(x)\p_x\zeta_\sharp f_2,
\end{aligned}
\eq
where
\bq\label{def:pq}
p=\frac{c^*-\zeta_\sharp {V^*}}{\zeta'},\quad q=-p\p_x(\zeta_\sharp B^*).
\eq
Since $B^*$ and $\zeta$ are odd and $V^*$ is even, it follows that $p$ and $q$ are even.  
  We apply $\zeta_*$ to \eqref{eq:v1} and $\zeta_\sharp$ to \eqref{eq:v2}, making use of \eqref{identity:zeta}.  
We rewrite the result as 
\begin{align}\label{eq:w12}
&\p_t w_1=\p_x\big(p(x)w_1\big)+|D_L|w_2,\\  \label{eq:w2}
&\p_t w_2=-\frac{g +q(x)}{\zeta'(x)}w_1+p(x)\p_xw_2,
\end{align}
 where
\bq\label{def:w12}
w_1=\zeta_*v_1,\quad w_2=\zeta_\sharp v_2
\eq
 are $L$-periodic. The Dirichlet-Neumann operator $G(\eta^*)$ in  \eqref{eq:v1} has thus been converted 
to the explicit Fourier multiplier $|D_L|$.
\subsection {Spectral modulational instability} 
 Modulational instability is the instability induced by long-wave perturbations. Therefore, we seek solutions of the linearized system \eqref{eq:w12}-\eqref{eq:w2} of the form $w_j(x, t)=e^{\ld t}e^{i\mu x}u_j(x)$, where $u_j(x)$ are $2\pi$-periodic. We assume $\mu=\frac{m_0}{n_0}$ is a {\it small rational number} and choose $L=n_0 2\pi \gg 2\pi$, so that $w_j(\cdot, t)$ is $L$-periodic. The following lemma avoids the Bloch transform. 
\begin{lemm}
For any  $f\in L^2(\T)$ we have 
\bq\label{conjugate:mu}
e^{-i\mu x}|D_L|(e^{i\mu x}f(x))=|D_L+\mu|f(x)=|D+\mu|f(x).
\eq
\end{lemm}
\begin{proof}
The first equality follows easily from the fact that $\wh{(e^{i\mu x}f)}^L(k)=\wh{f}^L(k-m_0)$. The second inequality follows from the general fact that if $f$ is $2\pi$-periodic, then for any  Fourier multiplier $b$ we have 
\bq\label{FM:L2pi}
b(D_{L})f(x)=b(D)f(x),\quad L=n_02\pi,
\eq
provided that they are well-defined as tempered distributions.  To prove \eqref{FM:L2pi}, let $\cF$ denote the Fourier transform and $\cF^{-1}$ the inverse Fourier transform, namely 
\[
\cF(f)(\xi)=\int_\Rr e^{-i\xi x}f(x)dx,\quad \cF^{-1}(f)(x)=\frac{1}{2\pi}\cF(f)(-x).
\]
For   $f\in L^2(\T)\subset L^2(\T_L)$ we have the inversion formula
\[
f(x)=\frac{1}{L}\sum_{k\in \Zz}e^{i\frac{2\pi}{L}kx}\wh{f}^L(k)\quad\text{in } L^2(\T_L)\subset \mathscr{S}'(\Rr),
\]
where $\mathscr{S}'(\Rr)$ is the space of tempered distributions. It follows that
\[
\cF(f)(\xi)=\frac{2\pi}{L}\sum_{k\in \Zz}\delta(\xi-\frac{2\pi}{L}k)\wh{f}^L(k)\in \mathscr{S}'(\Rr),
\]
where $\delta$ denotes the Dirac distribution centered at the origin. Consequently,
\[
\cF^{-1}(b\cF(f))(x)=\frac{1}{L}\sum_{k\in \Zz}e^{i\frac{2\pi}{L}kx}b\big(\frac{2\pi}{L}kx\big)\wh{f}^L(k)=b(D_L)f(x).
\]
Since $f$ is also $2\pi$-periodic, the preceding formula also holds for $L$ replaced by $2\pi$. Since the left side is independent of $L$, \eqref{FM:L2pi} follows.
\end{proof}
With the aid of \eqref{conjugate:mu}, from \eqref{eq:w12}-\eqref{eq:w2} we arrive at the pseudodifferential spectral problem
\bq\label{eq:u12}
\ld U=\mathcal{L}_{\mu, \eps}U:=
\begin{bmatrix}
p(i\mu+\p_x) +\p_x p & |D+\mu|\\
-\frac{g+q}{\zeta'} &p(i\mu+\p_x)
\end{bmatrix}
U,
\quad U=(u_1, u_2)^T,
\eq
where $U$ is $2\pi$-periodic.  The subscript $\eps$ indicates that the variable coefficients $p(x)$ and $q(x)$ depend upon $\eps$ through the Stokes wave. We regard $\mathcal{L}_{\mu, \eps}$ as a continuous operator from $(H^1(\T))^2$ to $(L^2(\T))^2$. The complex inner product of $L^2(\T)$ is denoted by
\[
(f_1, f_2)=\int_\T f_1(x)\ol{f_2}(x)dx.
\]
\begin{defi}[Spectral modulational instability]
If there exists a small  rational number $\mu$ such that the operator $\mathcal{L}_{\mu ,\eps}$ has an eigenvalue with positive real part, we say that the Stokes wave $(\eta^*, \psi^*, c^*, P^*=0)$ is subject to the spectral modulational (or Benjamin--Feir) instability. 
\end{defi}
 In what follows, we shall study \eqref{eq:u12} with $\mu$ being a small real number and prove that $\cL_{\mu, \eps}$ has an  eigenvalue with positive real part for all sufficiently small real numbers $\mu$, including in particular small rational numbers. We note that  $\mathcal{L}_{\mu, \eps}$ has the Hamiltonian structure
\bq\label{JK}
\mathcal{L}_{\mu, \eps}=J \mathcal{K}_{\mu, \eps}
\eq
where  $J=\begin{bmatrix} 0 & 1\\
-1 &0
\end{bmatrix}$ and  
\bq\label{def:Kmu}
\mathcal{K}_{\mu, \eps}=
\begin{bmatrix}
 \frac{g+q}{\zeta'} &-i\mu p - p\p_x\\
 i\mu p+p\p_x+\p_xp& |D+\mu|
\end{bmatrix}
\eq
is a  symmetric operator.  In particular, the adjoint of $\mathcal{L}_{\mu, \eps}$ is given by
\bq\label{L*}
\mathcal{L}_{\mu, \eps}^*=J\mathcal{L}_{\mu, \eps}J:(H^1(\T))^2\to (L^2(\T))^2.
\eq
Moreover, since 
\bq\label{spec:negmu}
\text{spec}_{L^2(\T)}(\mathcal{L}_{\mu, \eps})=\overline{\text{spec}_{L^2(\T)}(\mathcal{L}_{-\mu, \eps}}),
\eq
we lose no generality by considering  $\mathcal{L}_{\mu, \eps}$ for $\mu\in [0, \mez)$. 
 Furthermore, in system \eqref{sys:Stokes},  the change of variables 
 \bq\label{rescale:g}
 (\psi, c, P)\to (\sqrt{g} \psi, \sqrt{g}c, {P}/{g})
 \eq
 shows that we lose no generality by setting the gravity acceleration $g=1$.   The eigenvalues for the general case are obtained by multiplying the eigenvalues for the $g=1$ case  by $\sqrt{g}$.
 
 We end this section with the expansions in $\eps$ for the variable coefficients that appear in  $\mathcal{L}_{\mu, \eps}$.
  \begin{lemm}\label{lemm:expandpq}
We have the expansions
\begin{align}\label{zeta1}
&\zeta(x)=x+\eps \sin x+\eps^2 \sin(2x)+O_\eps(\eps^3),\\
 \label{expand:p}
&p(x)=1 -2\eps\cos x+\eps^2\Big(\tdm-2\cos(2x)\Big)+O_\eps(\eps^3),\\ \label{expand:q}
&q(x)=-\eps\cos x+\eps^2\big(1-\cos(2x)\big)+O_\eps(\eps^3),\\ \label{expand:qzeta}
&\frac{1+q(x)}{\zeta'}=1-2\eps \cos x+2\eps^2\big(1-\cos(2x)\big)+O_\eps(\eps^3).
\end{align}
\end{lemm}
 The notation $O_\eps$ indicates that the bound depends only on $\eps$. The proof of Lemma \ref{lemm:expandpq} makes use of the shape-derivative \eqref{shapederi} together with the expansion \eqref{z1} for the Riemann stretch, and is given in Appendix \ref{Appendix:expandpq}.

 \section{The operator $\mathcal{L}_{0, \eps}$}\label{section:Leps}
By virtue of Lemma \ref{lemm:expandpq}, in case $\eps=0$ the eigenvalue problem \eqref{eq:u12} reduces to
 \[
 \ld U=\mathcal{L}_{\mu, 0}U=
 \begin{bmatrix}
i\mu+\p_x  & |D+\mu|\\
-1& i\mu+\p_x
\end{bmatrix}
 U.
 \]
On the Fourier side, $\widehat U(k)\ne 0$ if and only if 
\[
\ld =  i[(\mu+k)\pm \sqrt{|\mu+k|}]  =:  i\omega^\pm_{k, \mu}.
\] 
Thus the spectrum is $\sigma(\mathcal{L}_{\mu, 0})=\{i\omega^\pm_{k, \mu}:  k\in \Zz\}\subset i\Rr$.   
Note that $\sigma(\mathcal{L}_{\mu, 0})$ is separated into the two parts 
$\sigma'(\mathcal{L}_{\mu, 0})\cup \sigma''(\mathcal{L}_{\mu, 0})$ where 
\[
\sigma'(\mathcal{L}_{\mu, 0})=\{i\om^+_{0, \mu}, i\om^-_{0, \mu}, i\om^-_{1, \mu}, i\om^+_{-1, \mu}\},   
\quad \sigma''(\mathcal{L}_{\mu, 0})=\{i\om^-_{-1, \mu}, i\om^+_{1, \mu}\}\cup\{ i\om_{k, \mu}^\pm: |k|\ge 2\}
\]
and each eigenvalue in $\sigma'(\mathcal{L}_{\mu, 0})$ is simple.   
 In case $\mu=\eps=0$,
\[
\omega^+_{0, 0}=\omega^-_{0, 0}=\omega^+_{-1, 0}=\omega^-_{1, 0}=0,
\]
so that the zero eigenvalue of $\mathcal{L}_{0, 0}$ has algebraic  multiplicity four 
and $\sigma''(\mathcal L_{0,0})$ is separated from zero.

Now we study the case when $\eps\ne 0$ is sufficiently small and $\mu=0$.  
By the semicontinuity of the separated parts of a spectrum (see IV-$\S$3.4 in \cite{Kato}) with respect to $\eps$, 
once again we have the separation 
\bq\label{sigma:L0}
\sigma(\mathcal{L}_{0, \eps})=\sigma'(\mathcal{L}_{0, \eps})\cup \sigma''(\mathcal{L}_{0, \eps}), 
\eq
 where the spectral subspace associated to the finite part $\sigma'(\mathcal{L}_{0, \eps})$ has dimension four. 
 We next prove  that zero is the only eigenvalue in $\sigma'(\mathcal{L}_{0, \eps})$ 
 by constructing  four {\it explicit} independent  eigenvectors in the generalized nullspace. 
\begin{theo}\label{kernel:L}
For  any sufficiently small $\eps$, zero is an eigenvalue of $\mathcal{L}_{0, \eps}$ with algebraic multiplicity four 
and geometric multiplicity two.  Moreover,  
\bq
\begin{aligned}
U_1&=(0, 1)^T \quad \text{and}\\
 \wt U_2&=\Big(\zeta_*\p_x\eta^*, \zeta_\sharp(\p_x\psi^*-B^*\p_x\eta^*)\Big)^T
 \end{aligned}
 \eq
are eigenvectors in the kernel, and 
\bq\label{def:U3U4}
\begin{aligned}
U_3&=\Big(\zeta_*\p_a\eta\vert_{(a, P)=(\eps, 0)}, \zeta_\sharp\big(\p_a\psi-B^*\p_a\eta\big)\vert_{(a, P)=(\eps, 0)}\Big)^T\\ 
U_4&=\Big(\zeta_*\p_P\eta\vert_{(a, P)=(\eps, 0)}, \zeta_\sharp\big(\p_P\psi-B^*\p_P\eta\big)\vert_{(a, P)=(\eps, 0)}\Big)^T
\end{aligned}
\eq
are generalized eigenvectors satisfying 
\bq\label{LU3U4}
\mathcal{L}_{0, \eps}U_3= -\p_a c\vert_{(a, P)=(\eps, 0)}\wt U_2,\quad \mathcal{L}_{0, \eps}U_4=-\p_P c\vert_{(a, P)=(\eps, 0)}\wt U_2-U_1.
\eq
 In \eqref{def:U3U4} and \eqref{LU3U4}, $(\eta, \psi, c)$ is any Stokes wave given by \eqref{expand:Stokes}. We also define the normalized second eigenvector 
\bq\label{def:U2}
U_2:=\frac{1}{\eps}\wt U_2 - \left(\frac{1}{2\pi\eps} \int_0^{2\pi} \wt U_2^{(2)}dx\right)U_1,   
\eq
where we write components $\wt U_2=(\wt U_2^{(1)}, \wt U_2^{(2)})^T$.  
Then $U_2$ is an eigenvector with mean zero.
\end{theo}
\begin{proof}
We have defined 
\bq\label{L0,eps}
\mathcal{L}_{0, \eps}=
\begin{bmatrix}
p\p_x +\p_x p & |D|\\
-\frac{g+q}{\zeta'} &p\p_x
\end{bmatrix},\quad g=1.
\eq
Firstly, it is clear that $U_1:=(0, 1)^T\in \text{ker}(\mathcal{L}_{0, \eps})$.   
Secondly, we differentiate \eqref{sys:Stokes} with respect to $x$ and  then evaluate at $(a, P)=(\eps, P^*=0)$ to obtain 
\[
\frac{\delta F_1(\eta^*, \psi^*, c^*)}{\delta(\eta, \psi)}(\p_x\eta^*, \p_x\psi^*)=0,  
\quad \frac{\delta F_2(\eta^*, \psi^*, c^*, P^*)}{\delta(\eta, \psi)}(\p_x\eta^*, \p_x\psi^*)=0,
\] 
where $(\eta^*, \psi^*, c^*)$ is given by \eqref{expand:star}. The identities  \eqref{dF1}, \eqref{dF2} and \eqref{identity} with $\ol\eta=\p_x\eta^*, \ol\psi=\p_x\psi^*$  then give
\[
\begin{aligned}
&\p_x\big((c^*-V^*)\p_x\eta^*\big)+G(\eta^*)\big(\p_x\psi^*-B^*\p_x\eta^*\big)=0,\\
&(c^*-V^*)\p_x^2\psi^* - B^*\p[(c^*-V^*)\p_x\eta^*] - B^*\p_xV^*\p_x\eta^* - g\eta^* = 0,
\end{aligned}
\]
so that 
\[
-\big(g+({V^*}-c^*)\p_x{B^*} \big)\p_x\eta^*+(c^*-{V^*})\p_x(\p_x\psi^*-B^*\p_x\eta^*)=0.
\]
Using \eqref{identity:zeta} with $f_1=\p_x\eta^*$ and $f_2=\p_x\psi^*-B^*\p_x\eta^*$, we deduce that 
\[   
\wt U_2:=\Big(\zeta_*\p_x\eta^*, \zeta_\sharp(\p_x\psi^*-B^*\p_x\eta^*)\Big)^T\in \text{ker}(\mathcal{L}_{0, \eps}).
\]
		Thirdly, we differentiate \eqref{sys:Stokes} with respect to $a$ 
		and then evaluate at $(a, P)=(\eps, 0)$  to obtain 
\begin{align*}
&\frac{\delta F_1(\eta^*, \psi^*, c^*)}{\delta(\eta, \psi)}(\p_a\eta\vert_{(a, P)=(\eps, 0)}, \p_a\psi\vert_{(a, P)=(\eps, 0)})=-\p_a c\vert_{(a, P)=(\eps, 0)}\p_x\eta^*,\\
& \frac{\delta F_2(\eta^*, \psi^*, c^*, P^*)}{\delta(\eta, \psi)}(\p_a\eta\vert_{(a, P)=(\eps, 0)}, \p_a\psi\vert_{(a, P)=(\eps, 0)})=-\p_a c\vert_{(a, P)=(\eps, 0)}\p_x\psi^*.
\end{align*}   
		Using \eqref{dF1}, \eqref{dF2} and \eqref{identity} with $(\ol\eta,\ol\psi) = (\p_a\eta^*, \p_a\psi^*)|_{(a,P)=(\eps,0)} $      
as well as \eqref{identity:zeta} with $(f_1,f_2) = (\p_a\eta^*, \p_a\psi^*-B^*\p_a\eta^*)$, we find that
\[
U_3:=\Big(\zeta_*\p_a\eta\vert_{(a, P)=(\eps, 0)}, \zeta_\sharp\big(\p_a\psi-B^*\p_a\eta\big)\vert_{(a, P)=(\eps, 0)}\Big)^T
\]
satisfies $\mathcal{L}_{0, \eps}U_3=-\p_a c\vert_{(a, P)=(\eps, 0)}\wt U_2$.    
Fourthly, 
differentiating \eqref{sys:Stokes} in $P$ and then evaluating at $(a, P)=(\eps, 0)$  yields 
\begin{align*}
&\frac{\delta F_1(\eta^*, \psi^*, c^*)}{\delta(\eta, \psi)}(\p_P\eta\vert_{(a, P)=(\eps, 0)}, \p_P\psi\vert_{(a, P)=(\eps, 0)})=-\p_P c\vert_{(a, P)=(\eps, 0)}\p_x\eta^*,\\
& \frac{\delta F_2(\eta^*, \psi^*, c^*, P^*)}{\delta(\eta, \psi)}(\p_P\eta\vert_{(a, P)=(\eps, 0)}, \p_P\psi\vert_{(a, P)=(\eps, 0)})=-\p_P c\vert_{(a, P)=(\eps, 0)}\p_x\psi^*-1,
\end{align*}
and hence
\[
U_4:=\Big(\zeta_*\p_P\eta\vert_{(a, P)=(\eps, 0)}, \zeta_\sharp\big(\p_P\psi-B^*\p_P\eta\big)\vert_{(a, P)=(\eps, 0)}\Big)^T
\]
satisfies $\mathcal{L}_{0, \eps}U_4=-\p_P c\vert_{(a, P)=(\eps, 0)}\wt U_2-U_1 = -U_1$.  
For the case of finite depth the term $\p_P c\vert_{(a, P)=(\eps, 0)}$ would not vanish  but for infinite depth it does.
Since $U_1$ and $\wt U_2$ are eigenvectors,   $U_3$ and $U_4$ are generalized eigenvectors.  Therefore we have  \eqref{LU3U4}.  
 Finally, note that $U_2$ has mean zero because $\wt U_2^{(1)}$ is an odd function due to the fact that both $\zeta$  
and $\p_x\eta^*$ are odd. 
\end{proof}
\begin{rema}  
The preceding proof works for both the finite and infinite depth cases. For the infinite depth case, we have the identity $G(B^*)\eta^*=-\p_xV^*$. See Remark 2.13 in \cite{ABZ1}.  It then follows directly from  \eqref{dF1}-\eqref{dF2} that  
\[
\frac{\delta F_1(\eta^*, \psi^*, c^*)}{\delta(\eta, \psi)}(\frac{1}{g}, 0)=0,\quad \frac{\delta F_2(\eta^*, \psi^*, c^*, Q=0)}{\delta(\eta, \psi)}(\frac{1}{g}, 0)=-1.
\]
 Consequently,  $U_4=\big(\zeta_*\frac{1}{g}, \zeta_\sharp(-\frac{1}{g}B^*)\big)^T=\big(\frac{1}{g}\zeta', -\frac{1}{g}\zeta_\sharp B^*\big)^T$ satisfies $\mathcal{L}_{0, \eps}U_4=-U_1$. This provides an alternative method to obtain $U_4$.
  \end{rema}
 For notational simplicity, we shall adopt the following abbreviations.
 \begin{nota}
 \begin{align*}
  & C=\cos x,\quad S=\sin x,\quad C_k=\cos(kx),\quad S_k=\sin(kx)\quad\text{for }k\in \{2, 3, 4,...\}.
 \end{align*}
\end{nota}
   \begin{coro}\label{coro:expandUj}
The components of $U_j$ defined  in Theorem \ref{kernel:L} have the following parity  and expansions.  
 \begin{align}\label{expand:U2}
 &U_2=
 \begin{bmatrix}
 \mathrm{odd}\\  \mathrm{even}
 \end{bmatrix}=
\begin{bmatrix}-S\\  C
\end{bmatrix} 
+\eps\begin{bmatrix}
-2S_2\\
C_2
\end{bmatrix}
+O_\eps(\eps^2),\\ \label{expand:U3}
&  U_3=
 \begin{bmatrix}
 \mathrm{ even}\\  \mathrm{odd}
 \end{bmatrix}=\begin{bmatrix} C\\ S
\end{bmatrix}
+\eps
\begin{bmatrix}
2C_2\\ S_2
\end{bmatrix}
+O_\eps(\eps^2),\\\label{expand:U4}
  & U_4=
 \begin{bmatrix}
  \mathrm{even}\\  \mathrm{odd}
 \end{bmatrix}=\begin{bmatrix}1\\ 0 \end{bmatrix}+ 
\eps
\begin{bmatrix}
C\\
-S
\end{bmatrix}
+O_\eps(\eps^2).
 \end{align}
 \end{coro}
 \begin{proof}
 From Theorem \ref{theo:Stokes} and Proposition \ref{prop:z12}, 
 it is clear that $\eta$ is even while both $\psi$ and $\zeta$ are odd.  
 It follows that $p$ and $q$, defined by \eqref{def:pq}, are even.  
 Consequently, the parity  properties stated in \eqref{expand:U2}, \eqref{expand:U3} and \eqref{expand:U4} follow.   
 Next we expand $U_j$ in powers of $\eps$. From \eqref{zeta1} and Taylor's formula, 
 for any function of the form $f=f^0+\eps f^1+O_\eps(\eps^2)$, we have 
\begin{align}\label{zeta:op1}
&\zeta_\sharp f(x)=f^0(x)+\eps\big(S\p_xf^0(x)+f^1(x)\big)+O_\eps(\eps^2),\\ \label{zeta:op2}
& \zeta_* f(x)=f^0(x)+\eps\big(Cf^0(x)+S\p_xf^0(x)+f^1(x)\big)+O_\eps(\eps^2).
\end{align}
On the other hand, if $f=\eps f^1+\eps^2 f^2+O_\eps(\eps^3)$, then
\begin{align}\label{zeta:op3}
&\zeta_\sharp f(x)=\eps f^1(x)+\eps^2\big(S\p_xf^1(x)+f^2(x)\big)+O_\eps(\eps^3),\\ \label{zeta:op4}
& \zeta_* f(x)=\eps f^1(x)+\eps^2\big(Cf^1(x)+S\p_xf^1(x)+f^2(x)\big)+O_\eps(\eps^3).
\end{align}
Using \eqref{expand:star} and \eqref{zeta:op1}-\eqref{zeta:op4}, and $B^*=\eps S+O_\eps(\eps^2)$ (see \eqref{B:app}) we find the expansion for $U_j$  as follows.
\begin{align*}
&\zeta_*\p_x\eta^*
=-\eps S-2\eps^2S_2+O_\eps(\eps^3)\\
&\zeta_\sharp(\p_x\psi^*-B^*\p_x\eta^*)=\eps  C+ \eps^2 C_2+O_\eps(\eps^3),\\
&\zeta_*\p_a\eta\vert_{(a, P)=(\eps, 0)}=C+\eps 2C_2+O_\eps(\eps^2)\\
& \zeta_\sharp\big(\p_a\psi-B^*\p_a\eta\big)\vert_{(a, P)=(\eps, 0)}=S+\eps S_2+O_\eps(\eps^2),\\
&\zeta_*\p_P\eta\vert_{(a, P)=(\eps, 0)}=1+\eps C+O_\eps(\eps^2),\\
& \zeta_\sharp\big(\p_P\psi-B^*\p_P\eta\big)\vert_{(a, P)=(\eps, 0)}=-\eps S+O_\eps(\eps^2).
\end{align*}
Note in particular that 
\bq\label{expand:wtU2}
\wt U_2=
 \begin{bmatrix}
 \mathrm{odd}\\  \mathrm{even}
 \end{bmatrix}=
\eps\begin{bmatrix}-S\\  C
\end{bmatrix} 
+\eps^2\begin{bmatrix}
-2S_2\\
C_2
\end{bmatrix}+O_\eps(\eps^3),
\eq
so that  $\int_0^{2\pi}\wt U^{(2)}_2dx=O_\eps(\eps^3)$ and the expansion for $U_2$ follows from \eqref{def:U2}. 
 \end{proof}
 Let $\mathcal{U}$ be the linear subspace of $(L^2(\T))^2$ spanned by the $(C^\infty(\T))^2$ vectors 
 $U_1\dots,U_4$ in Theorem \ref{kernel:L}.  
 Denote by  $\Pi$  the orthogonal  projection from $(L^2(\T))^2$ onto  the orthogonal complement 
 $\mathcal{U}^\perp$ of $\mathcal U$ in $(L^2(\T))^2$.
 The remainder of this section is devoted to the following  theorem in which the kernel and range of $\Pi \cL_{0, \eps}$ 
 are explicitly determined.  
 Recall that a linear operator is Fredholm if it is closed, has closed range of finite codimension, and 
 has a kernel of finite dimension. 
 \begin{theo}\label{theo:rangeL}
 For any sufficiently small $\eps$, $\Pi \cL_{0, \eps}:(H^1(\T))^2\to (L^2(\T))^2$ is a Fredholm operator with kernel $\cU$ and range $\cU^\perp$. 
  \end{theo}
   \begin{proof} 
  Since $\Pi \cL_{0, \eps}:(H^1(\T))^2\to (L^2(\T))^2$ is bounded, it is closed. We deduce from  \eqref{LU3U4} that
   \bq\label{KerL3}
   \mathcal{U}=\mathrm{Ker}(\mathcal{L}^2_{0, \eps})=\mathrm{Ker}(\mathcal{L}^m_{0, \eps})\quad\forall m\ge 3.
   \eq  
Thus  $\Pi \mathcal{L}_{0, \eps}V=0$  if and only if  $\mathcal{L}_{0, \eps}V\in \mathrm{Ker}(\mathcal{L}_{0, \eps}^2)$, or equivalently $V\in \mathrm{Ker}(\mathcal{L}_{0, \eps}^3)=\mathcal{U}$. In other words, $\mathrm{Ker}(\cL_{0, \eps})=\cU$. It remains to prove that $\Pi \cL_{0, \eps}$ maps onto $\cU^\perp$. 
This follows from the following two lemmas.
 \end{proof}
 The first lemma is a weaker statement. 
 \begin{lemm}
 We have
 \bq\label{key:onto}
  \ol{\mathrm{Ran}(\Pi \cL_{0, \eps})} = \cU^\perp, 
 \eq
 where  $\Pi \cL_{0, \eps}: (H^1(\T))^2\to \cU^\perp$.
 \end{lemm}
 \begin{proof}
 Since $\mathrm{Ran}(\Pi)=\cU^\perp$ is  a closed subspace,  by duality the identity \eqref{key:onto} is equivalent to 
 $ \mathrm{Ker}(\Pi)=\mathrm{Ker}(\cL_{0, \eps}^*\Pi)$, 
 where $\mathrm{Ker}(\Pi)=\cU\subset (H^\infty(\T))^2$. 
It is trivial that $\mathrm{Ker}(\Pi)\subset \mathrm{Ker}(\cL_{0, \eps}^*\Pi)$.  
Conversely suppose  $V\in  \mathrm{Ker}(\cL_{0, \eps}^*\Pi)$.  
Due to \eqref{L*} we have 
  \[
 \Pi V\in \mathrm{Ker}(\cL_{0, \eps}^*)=\mathrm{Ker}(J\cL_{0, \eps}J) 
 =\mathrm{Ker}(\cL_{0, \eps}J)=\mathrm{span}\{JU_1, J U_2\}.
 \]
 Thus $\Pi V=\beta_1 JU_1+\beta_2 JU_2$ for some $\beta_1,\ \beta_2\in \Cc$. Since $\Pi V\in \cU^\perp$,  $\beta_1 JU_1+\beta_2 JU_2$ is orthogonal to $U_3$ and $U_4$, so that  
  \bq\label{matrix:beta}
 \begin{bmatrix} 
 (JU_1, U_3) & (JU_2, U_3)\\
 (JU_1, U_4) & (JU_2, U_4)
 \end{bmatrix}
 \begin{bmatrix} \beta_1 \\ \beta_2 \end{bmatrix}=0.
 \eq
 Using the expansions for $U_j$ in \eqref{expand:Uj} we compute 
 \[
 (JU_1, U_3)=O_\eps(\eps^2),\quad (JU_1, U_4)=2\pi+ O_\eps(\eps^2),\quad (JU_2, U_3)=2\pi+O_\eps(\eps^2),\quad (JU_2, U_4)=O_\eps(\eps^2).
  \]
  Consequently, the determinant of the matrix in \eqref{matrix:beta} equals $-4\pi^2+O_\eps(\eps^2)$ which is nonzero  for all sufficiently small $\eps$. We conclude that $\beta_1=\beta_2=0$, yielding $\Pi V=0$ and hence $V\in \mathrm{Ker}(\Pi)$ as claimed.
 \end{proof}

 \begin{lemm}
$\mathrm{Ran}(\Pi \cL_{0, \eps})=\cU^\perp$.
\end{lemm}   
\begin{proof}
By virtue of \eqref{key:onto}, we only have to prove that 
$\mathrm{Ran}(\Pi \cL_{0, \eps})$ is closed in $(L^2(\T))^2$.   
It would be tempting to prove that $\Pi \cL_{0, \eps}$ is coercive. However, this is not the case as can be easily checked 
when $\eps=0$. Instead we appeal to a perturbative argument.   
According to Theorem 5.17, IV-$\S$5.2 in \cite{Kato}, the Fredholm property is stable under small perturbations.  
Therefore, it suffices to prove this property for $\eps=0$; that is, 
 the range of $\Pi \cL_{0, 0}$  equals $\cU^\perp$.   So now consider $\eps=0$. 
Given $F=(f_1, f_2)^T\in \cU^\perp$ we only have to prove that 
\bq\label{onto:L00}
F=\Pi \cL_{0, 0}V\quad\text{ for some } V\in (H^1(\T))^2.
\eq
Because $\eps=0$, the $U_j$ are precisely 
\bq
U_1=(0, 1)^T,\quad U_2=(-S, C)^T,\quad U_3=(C, S)^T,\quad U_4=(1, 0)^T. 
\eq
The $U_j$ are mutually orthogonal in $(L^2(\T))^2$, which implies that 
\bq\label{Pi0}
\Pi G=G-\sum_{j=1}^4\frac{(G, U_j)}{(U_j, U_j)}U_j\qquad\forall G\in (L^2(\T))^2.
\eq
Now for any $V=(v_1, v_2)^T\in (H^1(\T))^2$, we have
\[
\cL_{0, 0}V=\begin{bmatrix}\p_xv_1+|D|v_2\\ -v_1+\p_xv_2 \end{bmatrix}.
\]
We use \eqref{Pi0} to compute $\Pi \cL_{0, 0}V$.  
\begin{align*}
&(\cL_{0, 0}V, U_1)=\int_\T ( -v_1+\p_xv_2)dx=-\int_\T v_1dx,\\
&(\cL_{0, 0}V, U_4)=\int_\T (\p_xv_1+|D|v_2) dx=0,
\end{align*}
\begin{align*}
&(\cL_{0, 0}V, U_2)=\int_\T \{(\p_xv_1+|D|v_2)(-S)+(-v_1+\p_xv_2)C \}dx\\
&\qquad=\int_\T \{(v_1C-v_2S)+(-v_1C+v_2S)\}dx=0,\\
&(\cL_{0, 0}V, U_3)=\int_\T \{(\p_xv_1+|D|v_2)C+(-v_1+\p_xv_2)S\}dx\\
&\qquad=\int_\T \{(v_1S+v_2C)+(-v_1S-v_2C)\}dx=0.
\end{align*}
We obtain
\[
\Pi \cL_{0, 0}V= \cL_{0, 0}V+\Big(\frac{1}{2\pi}\int_\T v_1dx\Big) U_1=\begin{bmatrix}\p_xv_1+|D|v_2\\ -v_1+\frac{1}{2\pi}\int_\T v_1dx+\p_xv_2 \end{bmatrix},
\]
and hence \eqref{onto:L00} is equivalent to the system
\begin{align}\label{onto:eq1}
&\p_xv_1+|D|v_2=f_1,\\\label{onto:eq2}
 &-v_1+\frac{1}{2\pi}\int_\T v_1dx+\p_xv_2=f_2.
\end{align}
where we write  $F=(f_1, f_2)\in \cU^\perp$. 
It suffices to prove the existence of a solution $(v_1, v_2)^T\in (H^1(\T))^2$ of this system. 
From the orthogonality condition $(F, U_1)=0$ we have $\int_\T f_2dx=0$, 
and hence both sides of \eqref{onto:eq2} have mean zero.  
Thus upon differentiating \eqref{onto:eq2} we obtain the equivalent equation  
\bq\label{onto:eq3}
 -\p_xv_1+\p_x^2v_2=\p_xf_2.
 \eq
 Adding \eqref{onto:eq1} to \eqref{onto:eq3} yields an equation for $v_2$ alone, namely 
 \bq  \label{onto:eq4}
 \p_x^2v_2+|D|v_2=f_1+\p_xf_2. 
 \eq
 On the Fourier side this becomes 
\bq\label{onto:eq5}
 (-k^2+|k|)\widehat{v_2}(k)=\widehat{f_1}(k)+ik\widehat{f_2}(k)\quad\forall k\in \Zz.
 \eq
 Since $-k^2+|k|=0$ for $k\in \{-1, 0, 1\}$,  \eqref{onto:eq5} is solvable if and only if the following conditions hold
 \begin{align}\label{onto:cd1}
& \widehat{f_1}(0)=0,\\\label{onto:cd2}
 &\widehat{f_1}(1)+i\widehat{f_2}(1)=0,\\ \label{onto:cd3}
 &\widehat{f_1}(-1)-i\widehat{f_2}(-1)=0.
 \end{align}
 Condition \eqref{onto:cd1} is satisfied since $0=(F, U_4)=\int_\T f_2dx=\widehat{f_2}(0)$.  On the other hand,  the conditions $(F, U_2)=(F, U_3)=0$  can be written as 
 \begin{align*}
&-i[\widehat{f_1}(1)-\widehat{f_1}(-1)]+[\widehat{f_2}(1)+\widehat{f_2}(-1)]=0,\\
&[\widehat{f_1}(1)+\widehat{f_1}(-1)]+i[\widehat{f_2}(1)-\widehat{f_2}(-1)]=0.
 \end{align*}
Thus we  obtain both \eqref{onto:cd2} and \eqref{onto:cd3}. 
 We conclude that  the general periodic solution $v_2$ of \eqref{onto:eq4}  is 
 \bq\label{onto:v2}
 v_2(x)=b_0+b_{-1}e^{-ix}+b_1e^{ix}+\frac{1}{2\pi}\sum_{k\in \Zz\setminus\{-1, 0, 1\}} e^{ikx}\frac{\widehat{f_1}(k)+ik\widehat{f_2}(k)}{-k^2+|k|}.
 \eq
 Clearly $v_2\in H^1(\T)$.  Then, returning to \eqref{onto:eq1} and using the fact that $f_1$ has mean zero, we obtain 
  \bq\label{onto:v1}
 v_1(x)=a_0-\sign(D)v_2+\int_0^x f_1(x')dx'.
  \eq
  It is easy to deduce from \eqref{onto:v2} and \eqref{onto:v1} that $V\in (H^1(\T))^2$ if $ F\in (L^2(\T))^2$.  
  In fact, projecting $V$ onto $\cU^\perp$ fixes the constants $a_0,\ b_0,\ b_{-1}$ and $b_1$,  
 thereby yielding the unique solution  $\Pi V$  of \eqref{onto:L00} in $\cU^\perp$.
\end{proof}
  
 \section{Expansions of ${\bf A}_{\mu, \eps}$, ${\bf I}_\eps$ and $\det(\bf A_{\mu, \eps}-\ld {\bf I}_\eps)$}
 We define the matrices formed by $U_j$ and $\cL_{\mu, \eps}$, namely 
 \bq\label{def:AI}
 {\bf A_{\mu, \eps}}=\Big(\frac{(\mathcal{L}_{\mu, \eps}U_j, U_k)}{(U_k, U_k)}\Big)_{j, k=\overline{1, 4}},\quad {\bf I_{\eps}}=\Big(\frac{(U_j, U_k)}{(U_k, U_k)}\Big)_{j, k=\overline{1, 4}}.
 \eq
 Here and in what follows, we always consider $\mu\in [0, \mez)$.  
 \subsection
 {Expansions of ${\bf A_{\mu, \eps}}$ and ${\bf I}_\eps$}
  
In the following discussion, Fourier multipliers that act on $2\pi$-periodic functions are computed using the identities 
\bq
\begin{aligned}
&f(D)\cos (kx)=
\begin{cases}
if(k)\sin (kx)\quad\text{if $f$ is odd},\\
f(k)\cos (kx)\quad\text{if $f$ is even}
\end{cases},\\
&f(D)\sin (kx)=
\begin{cases}
-if(k)\cos (kx)\quad\text{if $f$ is odd},\\
f(k)\sin (kx)\quad\text{if $f$ is even}
\end{cases}.
\end{aligned}
\eq
We recall from Theorem \ref{kernel:L} and Corollary \ref{coro:expandUj} that the vectors $U_j$ are expanded as 
\bq\label{expand:Uj}
\begin{aligned}
&U_1=\begin{bmatrix} 0\\ 1\end{bmatrix},\quad U_2=
\begin{bmatrix}-S\\  C
\end{bmatrix} 
+\eps\begin{bmatrix}
-2S_2\\
C_2
\end{bmatrix}
+O_\eps(\eps^2),\\
& U_3=\begin{bmatrix} C\\ S
\end{bmatrix}
+\eps
\begin{bmatrix}
2C_2\\ S_2
\end{bmatrix}
+O_\eps(\eps^2),\quad U_4=\begin{bmatrix}1\\ 0 \end{bmatrix}+ 
\eps
\begin{bmatrix}
C\\
-S
\end{bmatrix}
+O_\eps(\eps^2).
\end{aligned}
\eq 
 In view of the identity $|k+\mu|=|k|+\mu \sign(k)$ for $|k|\ge 1$ and $\mu\in [0, \mez)$, we have 
  \[
  |D+\mu|u = \big(|D|+\mu\sign(D)\big)\big(u-\widehat{u}(0)\big)+\frac{1}{2\pi}\mu\widehat{u}(0) 
  =\big(|D|+\mu \sign(D)\big)u+\mu \frac{1}{2\pi}\int_\T udx.
  \]
Consequently, $\mathcal{L}_{\mu, \eps}$ can be decomposed as
\bq\label{decompose:L}
\mathcal{L}_{\mu, \eps}=\mathcal{L}_{0, \eps}+\mu (L_\eps^1 +L^\sharp),
\eq
where
\begin{align} \label{multtrig}
L_\eps^1=\begin{bmatrix}
ip& \mathrm{sign}(D)\\
0 &ip
\end{bmatrix}\quad\text{and}\quad L^\sharp\begin{bmatrix} u_1\\u_2\end{bmatrix}= \begin{bmatrix}  \frac{1}{2\pi}\int_\T u_2dx\\ 0\end{bmatrix}
\end{align}
are bounded on any Sobolev space $H^s(\T)$.  In the case of finite depth, there would also  be a term with $\mu^2$.  
Let us successively expand $\mathcal{L}_{\mu, \eps}U_j$ using the decomposition \eqref{decompose:L} 
together with  the expansion of $p$ from Lemma \ref{lemm:expandpq}.

i) $\mathcal{L}_{\mu, \eps}U_1$. We have $\mathcal{L}_{0, \eps}U_1=0$, $L^\sharp U_1=\begin{bmatrix} 1\\0 \end{bmatrix}$ and
\bq\label{L1U1}
\begin{aligned}
L^1_\eps U_1&= \begin{bmatrix}
0 \\ i
\end{bmatrix}
+
i\eps \begin{bmatrix}
0\\-2C
\end{bmatrix}
+ O_\eps(\eps^2).
\end{aligned}
\eq
(ii) $\mathcal{L}_{\mu, \eps}U_2$.  We have $\mathcal{L}_{0, \eps} U_2=0$, $L^\sharp U_2=0$ (because $U_2$ has mean zero) and
\bq\label{L1U2}
\begin{aligned}
L^1_\eps U_2&=
\begin{bmatrix}
-ipS+\sign(D)C\\
ipC
\end{bmatrix}
+\eps\begin{bmatrix}
-2ipS_2+\sign(D)C_2\\
ipC_2
\end{bmatrix}
+ O_\eps(\eps^2)\\
&=i\begin{bmatrix}
0\\C
\end{bmatrix}
+i\eps 
\begin{bmatrix}
0\\-1
\end{bmatrix}
+ O_\eps(\eps^2).
\end{aligned}
\eq
(iii) $\mathcal{L}_{\mu, \eps}U_3$. Since $\p_ac\vert_{(a, P)=(\eps, 0)}=\eps$, combining  \eqref{LU3U4}, \eqref{def:U2} and \eqref{expand:wtU2}  yields
\bq\label{L0U3:expand}
\mathcal{L}_{0, \eps} U_3=-\eps \wt U_2=-\eps\big(\eps U_2+O_\eps(\eps^3)U_1\big)=-\eps^2 U_2+O_\eps(\eps^4)U_1.
\eq
Noticing that the second components of $U_3$ and $U_4$ are odd, we have 
\bq\label{LsharpU3U4}
L^\sharp U_3=L^\sharp U_4=0.
\eq
On the other hand,
\bq
\begin{aligned}
L^1_\eps U_3
&=\begin{bmatrix}
ipC+\sign(D)S\\
ipS
\end{bmatrix}
+\eps \begin{bmatrix}
2ipC_2+\sign(D)S_2\\
ipS_2
\end{bmatrix}
+O_\eps(\eps^2)\\
&=i\begin{bmatrix}
0\\S
\end{bmatrix}
+i\eps \begin{bmatrix}
-1\\ 0
\end{bmatrix} +O_\eps(\eps^2).
 \end{aligned}
\eq
(iv) $\mathcal{L}_{\mu, \eps}U_4$. The fact that $\p_Pc\equiv 0$ combined with \eqref{LU3U4} yields
\[
\mathcal{L}_{0, \eps}  U_4=-\p_Pc\vert_{(a, P)=(\eps, 0)}\wt U_2-U_1=\begin{bmatrix} 0\\ -1\end{bmatrix}.
\]
 Taking \eqref{expand:p} into account, we compute 
\bq\label{L1U4}
\begin{aligned}
L^1_\eps U_4&=
\begin{bmatrix}
ip\\ 0
\end{bmatrix}+
\eps\begin{bmatrix}
ipC-\sign(D)S\\ -ipS
\end{bmatrix}
+O_\eps(\eps^2)=i\begin{bmatrix}
1\\ 0
\end{bmatrix}
+i\eps \begin{bmatrix}0\\-S
\end{bmatrix}+ O_\eps(\eps^2).
\end{aligned}
\eq 
 Now consider the various inner products.  Some of them vanish because of parity.
Since $U_1$ and $U_2$ are  
$\begin{bmatrix}
 \mathrm{odd}\\  \mathrm{even}
 \end{bmatrix}$
  and $p$ is even, we see that $L^1_\eps U_1$ and $L^1_\eps U_2$  are 
$  \begin{bmatrix}\mathrm{odd}\\\mathrm{even} \end{bmatrix}$. 
But $U_3$ and $U_4$ are 
 $\begin{bmatrix}
 \mathrm{even}\\  \mathrm{odd}
 \end{bmatrix},$
 so that  we find 
\[
(L^1_\eps U_1, U_3)=(L^1_\eps U_1, U_4)=(L^1_\eps U_2, U_3)=(L^1_\eps U_2, U_4)=0.
\]
We also recall that $L^\sharp U_1=\begin{bmatrix} 1\\ 0 \end{bmatrix}$ and $L^\sharp U_2=0$. 
 Therefore, denoting 
\bq
 {\bf M}_{jk}=(\mathcal{L}_{\mu, \eps}U_j, U_k)
 \eq
  we have
 \bq\label{M23,24}
 {\bf M}_{23}= {\bf M}_{24}=0
 \eq
 and 
\bq\label{M:row12}
\begin{aligned}
&{\bf M}_{11}=i\mu2\pi+\mu O_\eps(\eps^2),\quad {\bf M}_{12}=i\mu\eps (-2\pi)+\mu O_\eps(\eps^2),\\
& {\bf M}_{13}=\mu O_\eps(\eps^2),\quad  {\bf M}_{14}=\mu 2\pi+\mu O_\eps(\eps^2),\\
& {\bf M}_{21}=i\mu\eps(-2\pi)+\mu O_\eps(\eps^2),\quad {\bf M}_{22}=i\mu\pi+\mu O_\eps(\eps^2).
\end{aligned}
\eq

On the other hand, $\mathcal{L}_{0, \eps}U_3,~\mathcal{L}_{0, \eps}U_4=\begin{bmatrix} 
 \mathrm{odd}\\ \mathrm{even}
\end{bmatrix}$ and $L^1_\eps U_3,~L^1_\eps U_4=\begin{bmatrix} 
 \mathrm{even}\\ \mathrm{odd}
\end{bmatrix}$, yielding the fact that many  more inner products vanish: 
\bq
\begin{aligned}
&(\mathcal{L}_{0, \eps} U_3, U_3)=(\mathcal{L}_{0, \eps} U_3, U_4)=(\mathcal{L}_{0, \eps} U_4, U_3)=(\mathcal{L}_{0, \eps} U_4, U_4)=0,\\
&(L^1_\eps U_3, U_1)=(L^1_\eps U_3, U_2)=(L^1_\eps U_4, U_1)=(L^1_\eps U_4, U_2)=0.
\end{aligned}
\eq
We recall in addition that $L^\sharp U_3=L^\sharp U_4=0$,  $\mathcal{L}_{0, \eps}U_4=-U_1$ and $\mathcal{L}_{0, \eps}U_3=-\eps^2 U_2+O_\eps(\eps^4)U_1$ (see \eqref{L0U3:expand}). 
Consequently  
\bq\label{M31,42}
\begin{aligned}
&{\bf M}_{31}=(\mathcal{L}_{0, \eps}U_3, U_1)=-\eps^2(U_2, U_1)+O_\eps(\eps^4)(U_1, U_1)=O_\eps(\eps^4),\\
&{\bf M}_{41}=(\mathcal{L}_{0, \eps}U_4, U_1)=-(U_1, U_1)=-2\pi,\\
&{\bf M}_{42}=(\mathcal{L}_{0, \eps}U_4, U_2)=-(U_1, U_2)=0
\end{aligned}
\eq
due to $(U_1, U_2)=\int_0^{2\pi} U_2^{(2)}dx=0$. Moreover, 
\bq\label{M:row34}
\begin{aligned}
&{\bf M}_{32}=O_\eps(\eps^2),\quad {\bf M}_{33}=i\mu \pi +\mu O_\eps(\eps^2),\quad {\bf M}_{34}=i\mu\eps(-3\pi)+\mu O_\eps(\eps^2),\\
& {\bf M}_{43}=i\mu\eps (-\pi)+\mu O_\eps(\eps^2),\quad {\bf M}_{44}=i\mu 2\pi+\mu O_\eps(\eps^2).
\end{aligned}
\eq
This completes the expansion of the matrix {\bf M}.  For the case of finite depth, the algebra is considerably more complicated. Now by virtue of Corollary \ref{coro:expandUj} and the fact that $U_2$ has mean zero, we also have
\bq\label{UU}
\begin{aligned}
&(U_1, U_1)=2\pi,\quad (U_1, U_2)=0,\quad (U_1, U_3)=0,\quad (U_1, U_4)=0,\\
&(U_2, U_2)=2\pi+O_\eps(\eps^2),\quad (U_2, U_3)=0,\quad (U_2, U_4)=0,\\
&(U_3, U_3)=2\pi+O_\eps(\eps^2),\quad (U_3, U_4)=O_\eps(\eps^2),\quad (U_4, U_4)=2\pi+O_\eps(\eps^2).
\end{aligned}
\eq
Therefore, ${\bf I_{\eps}}=\big(\frac{(U_j, U_k)}{(U_k, U_k)}\big)_{j, k=\overline{1, 4}}$ is very simply expanded as
{\footnotesize
 \bq\label{matrixI}
{\bf I}_\eps=
\begin{bmatrix}
1 & 0 & 0 & 0\\
0 & 1 & 0 & 0\\
0 & 0 & 1 & O_\eps(\eps^2)\\
0 & 0 & O_\eps(\eps^2) & 1
\end{bmatrix}.
\eq
}

Combining this with \eqref{M23,24}, \eqref{M:row12}, \eqref{M31,42}, \eqref{M:row34} and \eqref{UU}, we also expand $({\bf A}_{\mu, \eps})_{jk}=\frac{({\bf M}_{\mu, \eps})_{jk}}{(U_k, U_k)}$
as
\bq\label{matrixA}
\begin{aligned}
&{\bf A}_{11}=i\mu +\mu O_\eps(\eps^2),\quad {\bf A}_{12}=-i\mu\eps +\mu O_\eps(\eps^2),\quad {\bf A}_{13}=\mu O_\eps(\eps^2),\quad {\bf A}_{14}=\mu +\mu O_\eps(\eps^2),\\
& {\bf A}_{21}=-i\mu\eps+\mu O_\eps(\eps^2),\quad {\bf A}_{22}=\mez i\mu+\mu O_\eps(\eps^2),\quad {\bf A}_{23}= {\bf A}_{24}=0,\\
&{\bf A}_{31}=O_\eps(\eps^4),\quad {\bf A}_{32}=O_\eps(\eps^2),\quad {\bf A}_{33}=\mez i\mu +\mu O_\eps(\eps^2),\quad {\bf A}_{34}=-\tdm i\mu\eps+\mu O_\eps(\eps^2),\\
&{\bf A}_{41}=-1,\quad {\bf A}_{42}=0,\quad {\bf A}_{43}=-\mez i\mu\eps +\mu O_\eps(\eps^2),\quad {\bf A}_{44}=i\mu +\mu O_\eps(\eps^2).
\end{aligned}
\eq
We can be more specific about ${\bf A}_{32}$.  
Indeed, because  $L^\sharp U_3=0$ and $(L^1_\eps U_3, U_2)=0$, we deduce from \eqref{L0U3:expand} that
\bq\label{A32}
{\bf A}_{32}=\frac{(\mathcal{L}_{0, \eps}U_3, U_2)}{(U_2, U_2)}=\frac{\big(-\eps^2 U_2+O_\eps(\eps^4)U_1, U_2\big)}{(U_2, U_2)}=-\eps^2.
\eq
We note that the exact coefficient of $\eps^2$ in ${\bf A}_{32}$ will be needed to determine the contribution of the main term $II_{10}$ in \eqref{expand:II} below. In \eqref{A32}, this is obtained by using the structure of the basis $\{U_j: j=\overline{1, 4}\}$ instead of expanding up to $O_\eps(\eps^3)$.

 Let us set ${\bf \wt A}_{jk}$ to be the  the {\it leading part} of the preceding expansion of ${\bf A}_{jk}$,  
 that is, without the remainder terms.  
 In particular, 
\[
{\bf \wt A}_{22}={\bf \wt A}_{33}, \quad {\bf \wt A}_{11}={\bf \wt A}_{44}  
\quad \text { and  }\quad  {\bf \wt A}_{13}={\bf \wt A}_{23}={\bf \wt A}_{24}={\bf \wt A}_{31}={\bf \wt A}_{42}=0.
\]
 Combining this with  \eqref{matrixI}, \eqref{matrixA} and \eqref{A32}, we can write the whole matrix as  
\bq\label{final:A-I}
\begin{aligned}
&{\bf A}_{\mu, \eps}-\ld{\bf I}_\eps=\\
&\begin{bmatrix}
{\bf \wt A}_{11}-\ld+\mu O_\eps (\eps^2) &{\bf \wt A}_{12}+\mu O_\eps (\eps^2)  & \mu O_\eps(\eps^2) &{\bf \wt A}_{14}+\mu O_\eps (\eps^2)\\
{\bf \wt A}_{21}+\mu O_\eps(\eps^2)&{\bf \wt A}_{22}-\ld+\mu O_\eps(\eps^2) &0 &  0\\
O_\eps(\eps^4) & -\eps^2 &{\bf \wt A}_{33}-\ld+\mu O_\eps( \eps^2)  & {\bf \wt A}_{34}+\mu O_\eps( \eps^2)+\ld O_\eps(\eps^2)\\
-1 &  0 &{\bf \wt A}_{43} +\mu O_\eps(\eps^2)+\ld O_\eps(\eps^2)&{\bf \wt A}_{44}-\ld+\mu O_\eps( \eps^2)
\end{bmatrix}.
\end{aligned}
\eq

\subsection{Expansion of $\det({\bf A}_{\mu, \eps}-{\bf I}_\eps)$}
We write out the individual terms of the determinant  of $({\bf A}_{\mu, \eps}-{\bf I}_\eps)$.  
We  observe that in \eqref{final:A-I} the only entries {\it without} $\mu$ or $\ld$ are the $(3, 1)$, $(3, 2)$ and $(4, 1)$ entries.   
So let us consider  those terms.  
Only the $(3, 2)$ and $(4, 1)$ entries are multiplied by each other in the terms 
$(3, 2)(4, 1)(j, k)(j', k')$ where $j, j'\in \{1, 2\}$ and $k, k'\in \{3, 4 \}$.   
Because the $(2, 3)$ and $(2, 4)$ entries are identically zero, the terms $(3, 2)(4, 1)(j, k)(j', k')$ vanish.   
We deduce that each term in  $\det{\bf A}_{\mu, \eps}-{\bf I}_\eps$ is at most $O(\mu^3+|\ld|^3)$.  

{\it Taking $\eps$ into account, we shall treat  $O(\mu^4+|\ld|^4)$ and $O(\eps^3)$ terms as remainders.} 
Evaluating  $\det({\bf A}_{\mu, \eps}-{\bf I}_\eps)$ with respect to the second row yields the expansion 
\bq\label{expand:detA}
\begin{aligned}
&\det({\bf A}_{\mu, \eps}-I_\eps)=\\
&-[{\bf \wt A}_{21}+\mu O_\eps(\eps^2)][{\bf \wt A}_{12}+\mu O_\eps(\eps^2)][{\bf \wt A}_{33}-\ld+\mu O_\eps(\eps^2)][{\bf \wt A}_{44}-\ld+\mu O_\eps(\eps^2)]\\
&+[{\bf \wt A}_{21}+\mu O_\eps(\eps^2)][{\bf \wt A}_{12}+\mu O_\eps(\eps^2)][{\bf \wt A}_{43}+\mu O_\eps(\eps^2)+\ld O_\eps(\eps^2)][{\bf \wt A}_{34}+\mu O_\eps(\eps^2)+\ld O_\eps(\eps^2)]\\
&+[{\bf \wt A}_{21}+\mu O_\eps(\eps^2)][\mu O_\eps(\eps^2)][-\eps^2][{\bf \wt A}_{44}-\ld+\mu O_\eps(\eps^2)]\\
&-[{\bf \wt A}_{21}+\mu O_\eps(\eps^2)][{\bf \wt A}_{14}+\mu O_\eps(\eps^2)][-\eps^2][{\bf \wt A}_{43} +\mu O_\eps(\eps^2)+\ld O_\eps(\eps^2)]\\
&+[{\bf \wt A}_{22}-\ld+\mu O_\eps(\eps^2)][{\bf \wt A}_{11}-\ld+\mu O_\eps(\eps^2)][{\bf \wt A}_{33}-\ld+\mu O_\eps(\eps^2)][{\bf \wt A}_{44}-\ld+\mu O_\eps(\eps^2)]\\
&-[{\bf \wt A}_{22}-\ld+\mu O_\eps(\eps^2)][{\bf \wt A}_{11}-\ld+\mu O_\eps(\eps^2)] [{\bf \wt A}_{34}+\mu O_\eps(\eps^2)+\ld O_\eps(\eps^2)][{\bf \wt A}_{43} +\mu O_\eps(\eps^2)+\ld O_\eps(\eps^2)]\\
&-[{\bf \wt A}_{22}-\ld+\mu O_\eps(\eps^2)][\mu O_\eps(\eps^2)][O_\eps(\eps^4)][{\bf \wt A}_{44}-\ld+\mu O_\eps(\eps^2)]\\
&+[{\bf \wt A}_{22}-\ld+\mu O_\eps(\eps^2)][\mu O_\eps(\eps^2)][-1][ {\bf \wt A}_{34}+\mu O_\eps(\eps^2)+\ld O_\eps(\eps^2)]\\
&+[{\bf \wt A}_{22}-\ld+\mu O_\eps(\eps^2)][{\bf \wt A}_{14}+\mu O_\eps(\eps^2)][O_\eps(\eps^4)][{\bf \wt A}_{43} +\mu O_\eps(\eps^2)+\ld O_\eps(\eps^2)]\\
&-[{\bf \wt A}_{22}-\ld+\mu O_\eps(\eps^2)][{\bf \wt A}_{14}+\mu O_\eps(\eps^2)][-1][{\bf \wt A}_{33}-\ld+\mu O_\eps(\eps^2)]\\
&=T_1+...+T_{10},  
\end{aligned}
\eq
respectively.   In order to simplify the subsequent exposition, we introduce the following notation for polynomials of $(\mu, \ld)$:
\bq\label{def:Pi34}
\Pi_3(\mu, \ld)=a_0\mu^3+a_1\mu^2\ld+a_2\mu\ld^2,\quad \Pi_4(\mu, \ld)=a_0\mu^4+a_1\mu^3\ld+a_2\mu^2\ld^2+a_3\mu\ld^3+a_4\ld^4,
\eq
where the $a_j$ may depend on $\eps$. 
We emphasize that $\Pi_3(\mu, \ld)$ does not have a $\ld^3$ term.  
Examining the explicit formulas for ${\bf \wt A}_{jk}$, we find that 
\[
\begin{aligned}
&T_1=O_\eps(\eps^2)\Pi_4(\mu, \ld),\quad T_2=O_\eps(\eps^4)\Pi_4(\mu, \ld),\quad T_3=O_\eps(\eps^5)\Pi_3(\mu, \ld),\\
&T_4=O_\eps(\eps^4)\Pi_3(\mu, \ld),\quad T_5=O_\eps(1)\Pi_4(\mu, \ld),\quad T_6=O_\eps(\eps^2)\Pi_4(\mu, \ld),\\
&T_7=O_\eps(\eps^6)\Pi_3(\mu, \ld),\quad T_8=O_\eps(\eps^3)\Pi_3(\mu, \ld),\quad T_9=O_\eps(\eps^5)\Pi_3(\mu, \ld),\\
&T_{10}=\mu(\mez i\mu-\ld)^2+\mu(\mez i\mu -\ld)^2O_\eps(\eps^2)+\mu^2(\mez i\mu -\ld)O_\eps(\eps^2)+O_\eps(\eps^4)\Pi_3(\mu, \ld).
\end{aligned}
\]
In other words, $T_{10}$ is the only main term. Therefore we have proved
\begin{prop}\label{prop:detA}
\bq\label{formdetA}
\begin{aligned}
\det({\bf A}_{\mu, \eps}-\ld {\bf I}_\eps)&=(\mez i\mu-\ld)^2\mu+\mu(\mez i\mu -\ld)^2O_\eps(\eps^2)+\mu^2(\mez i\mu -\ld)O_\eps(\eps^2)+O_\eps(\eps^3)\Pi_3(\mu, \ld)\\
&\quad+O_\eps(1)\Pi_4(\mu, \ld).
\end{aligned}
\eq
\end{prop}
It will turn out that the precise coefficients of $\eps^2$ in the $O_\eps(\eps^2)$ terms 
in \eqref{formdetA} are not needed, thanks to  presence of the factor $(\mez i\mu -\ld)$.
  \section{Perturbation of eigenfunctions due to sidebands} 
 The small parameters involved in our proof are $\ld$, $\mu$ and $\eps$, where we recall that $\mu \in [0, \mez)$.
As above, the notation $O(\eps^k)$ signifies smooth functions $f(\ld, \mu, \eps)$  
bounded  by $C|\eps|^k$ for small $(\ld, \mu, \eps)$.  
In case $f$ depends only on $\eps$ we write $O(\eps^k)=O_\eps (\eps^k)$. 

Moreover, the notation $O(\mu^m+|\ld|^m)$, for $m\in \{0, 1, \dots\}$, signifies smooth functions $f(\ld, \mu, \eps)$ 
that satisfy both (i) $f(\ld, \mu, \eps)\le C(\mu^m +|\ld|^m)$ for small $(\ld, \mu, \eps)$ 
and (ii) $\mu^{-m}f(\ld, \mu, \eps)=\wt f(\frac{\ld}{\mu}, \mu, \eps)$ for some smooth function $\wt f$.

 \subsection{Lyapunov-Schmidt method}  
 Our ultimate goal is to study the eigenvalue problem  $\mathcal{L}_{\mu,\eps}U=\ld U$ for fixed 
 small parameters $\eps$ and $\mu\ge 0$.    Recall from Section 4 that $\cU$, the linear subspace of $(L^2(\T))^2$ spanned by the vector $U_j$ 
given in Theorem \ref{kernel:L}, is the  generalized eigenspace associated to the eigenvalue $\ld=0$ of $\cL_{0, \eps}$.  
Permitting $\mu>0$ we seek  generalized eigenvectors bifurcating from $U_j$. By \cite{Kato} there exists a four dimensional nullspace of $\cL_{\mu, \eps}$ for small $\mu$.
The Lyapunov-Schmidt method splits the eigenvalue problem into finite and infinite dimensional parts. 
In our case, there are at least two difficulties (i) the generalized kernel $\cU$ of $\cL_{0, \eps}$ is strictly larger 
than its kernel and (ii) $\cL_{0, \eps}$ is neither self-adjoint nor skew-adjoint.  
We resolve these difficulties by using Theorem \ref{theo:rangeL}.

Recalling that $\Pi$ denotes the orthogonal  projection from $L^2(\T)^2$ onto $\mathcal{U}^\perp$ 
with respect to the $(L^2(\T))^2$ inner product, we want to solve the system 
  \begin{align}\label{PL1}
  &\Pi (\mathcal{L}_{\mu, \eps}-\ld \mathrm{Id})U=0,\\ \label{PL2}
   &(\mathrm{Id}-\Pi) (\mathcal{L}_{\mu, \eps}-\ld \mathrm{Id})U=0.
  \end{align}
 If we seek solutions of the form $U=\sum_{\alpha=1}^4 \alpha_jU_j+W$ with $W\in  H^1(\T)^2\cap \mathcal{U}^\perp$, \eqref{PL1} is equivalent to 
  \bq\label{PL1:2}
  \Pi (\mathcal{L}_{\mu, \eps}-\ld \mathrm{Id})\big(\sum_{j=1}^4\alpha_j U_j +W\big)=0. 
  \eq
 By the linearity in $\alpha_j$,  clearly $W=\sum_{j=1}^4\alpha_jW_j$, where each {\it sideband function} $W_j$ solves 
   \bq\label{PL1:3}
 \Pi (\mathcal{L}_{\mu, \eps}-\ld \mathrm{Id})(U_j +W_j)=0 
  \eq 
 for $j=1,2,3,4$. According to Theorem \ref{theo:rangeL}, $\mathrm{Ker}(\Pi \cL_{0, \eps})=\cU$, so that $\Pi \cL_{0, \eps}U_j=0$ and \eqref{PL1:3} can be written in greater detail as 
   \bq\label{PL1:4}
T_{\ld, \mu, \eps}W_j:= \left[\Pi \mathcal{L}_{0, \eps}+ \Pi \big(\mu (L^1_\eps+L^\sharp)-\ld \mathrm{Id}\big)\right]W_j=-\mu \Pi (L^1_\eps+L^\sharp) U_j.
  \eq
By Theorem \ref{theo:rangeL} the operator $\Pi\cL_{0, \eps}: (H^1(\T))^2\cap \cU^\perp\to \cU^\perp\subset (L^2(\T))^2$ 
is an isomorphism. So its inverse is also bounded by virtue of the open mapping theorem. Let us denote
\bq\label{def:Xi}
\Xi_\eps=(\Pi\cL_{0, \eps})^{-1}:\cU^\perp\to (H^1(\T))^2\cap \cU^\perp  
\eq
 and call it the {\it inverse operator}.    Then 
\[
\mathrm{Id}-\Xi_\eps T_{\ld, \mu, \eps}  =  -\Xi_\eps \Pi \big(\mu (L^1_\eps+L^\sharp)-\ld \mathrm{Id}\big).
\]
 {\it Thus for each small $\eps$,  if  $\mu$ and $\ld$ are sufficiently small}, 
 then the Neumann series $\sum_{m=0}^\infty (\mathrm{Id}-\Xi_\eps T_{\ld, \mu, \eps})^m$ converges 
 as an operator on $ (H^1(\T))^2\cap \cU^\perp$.  
 Therefore $\Xi_\eps T_{\ld, \mu, \eps}$ is invertible  from $ (H^1(\T))^2\cap \cU^\perp$ 
 onto $H^1(\T)^2\cap \cU^\perp$.  Its inverse is 
 \[
(\Xi_\eps T_{\ld, \mu, \eps})^{-1}  =  \big(\mathrm{Id}-(\mathrm{Id}-\Xi_\eps T_{\ld, \mu, \eps})\big)^{-1}  
=  \sum_{m=0}^\infty (\mathrm{Id}-\Xi_\eps T_{\ld, \mu, \eps})^m  
=  \sum_{m=0}^\infty(-1)^m\big[\Xi_\eps \Pi \big(\mu (L^1_\eps+L^\sharp)-\ld \mathrm{Id}\big)\big]^m.
 \]
Then applying $\Xi_\eps$ followed by $(\Xi_\eps T_{\ld, \mu, \eps})^{-1}$ to  \eqref{PL1:4}, we obtain
 \bq\label{Neumann:Wj}
   \begin{aligned}
   W_j  &=-\mu \sum_{m=0}^\infty(-1)^m\big[\Xi_\eps\Pi \big(\mu (L^1_\eps+L^\sharp)-\ld \mathrm{Id}\big)\big]^m\Xi_\eps\Pi (L^1_\eps+L^\sharp) U_j\in  (H^1(\T))^2\cap \cU^\perp.
   \end{aligned}
   \eq 
 This is the solution of \eqref{PL1:3}.   In particular,  it is clear that  
   \bq\label{est:W}
   \|W_j\|_{(H^1(\T))^2} =\mu O(1).
   \eq
We note that $U\ne 0$ if and only if $[\alpha_j]_{j=1}^4\ne 0$.  Substituting $W=\sum_{j=1}^4\alpha_jW_j$ into \eqref{PL2} gives 
  \bq\label{PL2:2}
  \sum_{j=1}^4\alpha_j(\mathrm{Id}-\Pi) (\mathcal{L}_{\mu, \eps}-\ld \mathrm{Id})(U_j+W_j)=0.
  \eq
 Now for any $V$, $(\mathrm{Id}-\Pi)V=0$ if and only if $V\in \mathcal{U}^\perp$.  
 Thus \eqref{PL2:2}  has a nontrivial solution $[\alpha_j]_{j=1}^4$ if and only if 
  \bq\label{eq:det}
  \det \Big((\mathcal{L}_{\mu, \eps}-\ld \mathrm{Id})(U_j+W_j), U_k\Big)_{jk}=0,
  \eq
  where $(W_j, U_k)=0$ for all $j, k=1,...,4$.   
For the sake of normalization, \eqref{eq:det}  is equivalent to 
\bq\label{def:cP}
\cP(\ld; \mu, \eps):=\det({\bf A}_{\mu, \eps}-\ld {\bf I}_\eps+{\bf B}_{\mu, \eps})=0,
\eq 
where  the {\it sideband matrix} is
\bq\label{def:B}
 {\bf B}_{\mu, \eps}=\Big(\frac{(\mathcal{L}_{\mu, \eps}W_j, U_k)}{(U_k, U_k)}\Big)_{j, k=\overline{1, 4}}.
 \eq
 Therefore we have proved
 \begin{prop}\label{prop:instab}
The Stokes wave $(\eta^*, \psi^*, c^*, P*=0)$ is modulationally unstable if there exists a small  rational number  $\mu>0$  such that \eqref{def:cP} has a sufficiently small root $\ld$ with positive real part. 
\end{prop}
\subsection{Analysis of the sideband matrix}
It follows from \eqref{est:W} that ${\bf B_{\mu, \eps}}= O(\mu)$ .   
In this subsection, we derive more precise estimates for ${\bf B}_{\mu, \eps}$. 
\begin{lemm}
\begin{align}\label{JLJU1}
&J\mathcal{L}_{0, \eps}JU_1=\begin{bmatrix}-1\\0 \end{bmatrix}+\eps  \begin{bmatrix}2C\\ -2S\end{bmatrix}+O_\eps(\eps^2),\\\label{JLJU2}
&J\mathcal{L}_{0, \eps}JU_2=\eps\begin{bmatrix}3C_2\\-4S_2\end{bmatrix}+O_\eps(\eps^2),\\ \label{JLJU3}
&J\mathcal{L}_{0, \eps}JU_3=\eps\begin{bmatrix}3S_2\\4C_2\end{bmatrix}+O_\eps(\eps^2),\\ \label{JLJU4}
&J\mathcal{L}_{0, \eps}JU_4=\eps \begin{bmatrix} 2S\\2C\end{bmatrix}+O_\eps(\eps^2).
\end{align}
In particular, 
\bq\label{boundJLJ}
\Pi J\mathcal{L}_{0, \eps}JU_k=O_\eps(\eps)\quad\forall k.
\eq
\end{lemm}
\begin{proof}
The operator $J$ is the skew-symmetric matrix in the Hamiltonian form \eqref{JK}.  
The expansions \eqref{JLJU1}-\eqref{JLJU4} are obtained by straightforward calculations using  
Lemma \ref{lemm:expandpq}.   
As for \eqref{boundJLJ} we note that $(U_m, U_n)=O_\eps(\eps^2)$ for $m\ne n$, so that
 \bq\label{expand:Pi}
\Pi V=V-\sum_{m=1}^4 \frac{(V, U_m)}{(U_m, U_m)}U_m+O(\eps^2)\qquad \forall\ V\in L^2(\T)^2.
\eq
 We put $V = J\mathcal L_{0,\eps} J U_k$.  Then \eqref{boundJLJ} is obvious for $k=2,3,4$.  
As for $k=1$, we use \eqref{expand:Pi}, \eqref{JLJU1} and \eqref{expand:U4} to find that 
the term independent of $\eps$ vanishes.  So \eqref{boundJLJ} follows. 
\end{proof}
\begin{lemm}\label{lemm:parity}
The following parity properties hold. 

(a)  The projection $\Pi$ preserves the parity. That is,
\bq\label{parity:Pi}
\Pi V=\begin{bmatrix} \mathrm{odd}\\ \mathrm{even} \end{bmatrix}\, \text{if }\,  V=\begin{bmatrix} \mathrm{odd}\\ \mathrm{even} \end{bmatrix},\quad\text{and}\quad \Pi V=\begin{bmatrix} \mathrm{even}\\ \mathrm{odd} \end{bmatrix}\, \text{if }\,  V=\begin{bmatrix} \mathrm{even}\\ \mathrm{odd} \end{bmatrix}.
\eq
(b) The inverse operator $\Xi_\eps=(\Pi\mathcal{L}_{0, \eps})^{-1}:\cU^\perp\to (H^1(\T))^2\cap \cU^\perp$  switches the parity. That is,
\bq\label{parity:Xi}
\Xi_\eps F=\begin{bmatrix} \mathrm{even}\\ \mathrm{odd} \end{bmatrix}\, \text{if }\,  F=\begin{bmatrix} \mathrm{odd}\\ \mathrm{even} \end{bmatrix},\quad\text{and}\quad \Xi_\eps F=\begin{bmatrix} \mathrm{odd}\\ \mathrm{even} \end{bmatrix}\, \text{if }\,  F=\begin{bmatrix} \mathrm{even}\\ \mathrm{odd} \end{bmatrix}.
\eq
\end{lemm}
\begin{proof}
(a) By Gram-Schmidt  orthonormalization we obtain four mutually orthogonal vectors $ U^\sharp_j$ 
that span $\cU$ such that each $ U^\sharp_j$ has the same parity as $U_j$.   
Then \eqref{parity:Pi}  follows at once from the formula  $\Pi V=V-\sum_{j=1}^4 (V, U^\sharp_j)U^\sharp_j$ 
and the parity of the $U^\sharp_j$. 

(b) Let us prove the first assertion in \eqref{parity:Xi}, as the second one follows analogously.   
Assuming $F=(\mathrm{odd}, \mathrm{even})^T\in \cU^\perp$, 
we will prove that $V=\Xi_\eps F=(\mathrm{even}, \mathrm{odd})^T$, where $V\in \cU\cap  (H^1(\T))^2$.  
To that end, for any function $f:\T\to \Cc$ we denote its even and odd parts by superscripts: 
\[
f^e(x)=\mez(f(x)+f(-x)),\quad f^o(x)=\mez(f(x)-f(-x)).
\]
Then we decompose $V=(v_1, v_2)^T$ as
\[
V=V'+V'',\quad V'=(v_1^e, v_2^o)^T,\quad V''=(v_1^o, v_2^e)^T. 
\]
It remains to prove that $V''=0$.   
Clearly $\cL_{0, \eps}$ switches the parity, and hence  so does $\Pi\cL_{0, \eps}$ in view of \eqref{parity:Pi}.   
In particular, $\Pi\cL_{0, \eps}V'=(\mathrm{odd}, \mathrm{even})^T$ and  $\Pi\cL_{0, \eps}V''=(\mathrm{even}, \mathrm{odd})^T$. Since $\Pi\cL_{0, \eps}V'+\Pi\cL_{0, \eps}V''=\Pi\cL_{0, \eps}V=F=(\mathrm{odd}, \mathrm{even})^T$, 
we must have $\Pi\cL_{0, \eps}V''=0$. Thus $V''\in \mathrm{Ker}(\Pi\cL_{0, \eps})=\cU$ by virtue of Theorem \ref{theo:rangeL}.   
In order to conclude that $V''=0$, it remains to prove $V''\in \cU^\perp$. Indeed,  we recall  that $U_1$ and $U_2$ are $(\mathrm{odd}, \mathrm{even})^T$, whereas $U_3$ and $U_4$ are $(\mathrm{even}, \mathrm{odd})^T$.   
In particular, $V''$ has opposite parity compared to $U_3$ and $U_4$, so that $(V'', U_3)=(V'', U_4)=0$.   
On the other hand,  for $j=1, 2$,  writing  the components as $U_j=(u_j^{(1)}, u_j^{(2)})$ where $u_j^{(1)}$ is odd and $u_j^{(2)}$ is even, 
the simple change of variables $-x\mapsto x$ implies that 
\begin{align*}
(V'', U_j)&=\mez\int_\T \big(v_1(x)-v_1(-x)\big)u_j^{(1)}(x)dx+\mez\int_\T \big(v_2(x)+v_2(-x)\big)u_j^{(2)}(x)dx\\
&=\mez\int_\T \big(v_1(x)+v_1(x)\big)u_j^{(1)}(x)dx+\mez\int_\T \big(v_2(x)+v_2(x)\big)u_j^{(2)}(x)dx\\
&=(V, U_j) = 0 
\end{align*}
because  $V\in \cU^\perp$.   Thus $V''\in\cU^\perp$.  This completes the proof of \eqref{parity:Xi}.
\end{proof}
\begin{lemm}\label{lemm:B:1}
Let $\Pi_2(\mu, \ld)$ denote any polynomial of the form $a_0\mu^2+a_1\mu \ld$.  We have
 \begin{align}\label{estB:lowerleft}
& ({\bf B_{\mu, \eps}})_{jk}= O(\mu^2+|\ld|^2)\quad\text{for}~j\in \{3, 4\},~k\in \{1, 2\},\\ \label{estB:1k}
 &({\bf B_{\mu, \eps}})_{1k}= \mu O_\eps (\eps^2)+O(\mu^2+|\ld|^2)\quad\forall k,\\  \label{estB:21}
 &({\bf B_{\mu, \eps}})_{21}=\frac{3i}{4} \mu \eps+\mu O_\eps(\eps^2)+O(\mu^2+|\ld|^2),\\ \label{estB:22}
  &({\bf B_{\mu, \eps}})_{22}= \mu O_\eps(\eps^2)+O(\mu^2+|\ld|^2),\\ \label{estB:23}
  &({\bf B_{\mu, \eps}})_{23}=-\frac{1}{8}\mu^2+\eps \Pi_2(\mu, \ld)+O(\mu^3+|\ld|^3),\\ \label{estB:24}
  & ({\bf B_{\mu, \eps}})_{24}=\eps \Pi_2(\mu, \ld)+O(\mu^3+|\ld|^3),\\\label{estB:33}
  &({\bf B_{\mu, \eps}})_{33}= \mu O_\eps(\eps^2)+O(\mu^2+|\ld|^2),\\ \label{estB:34}
  &({\bf B_{\mu, \eps}})_{34}=\frac{i}{2}\mu \eps+\mu O_\eps(\eps^2)+O(\mu^2+|\ld|^2),\\ \label{estB:4k}
& ({\bf B_{\mu, \eps}})_{4k}=\mu O_\eps (\eps^2)+O(\mu^2+|\ld|^2)\quad\text{for } k\in \{3, 4\}.
 \end{align}
\end{lemm}
\begin{rema}
It is crucial to the proof of instability in Section 7 that the coefficient of the leading term $-\frac{1}{8}\mu^2$ in $({\bf B_{\mu, \eps}})_{23}$ is negative. 
\end{rema}
\begin{proof}[Proof of Lemma \ref{lemm:B:1}]
We recall the definition \eqref{def:B} of ${\bf B}_{\mu,\eps}$.
Because  $(U_k, U_k)=2\pi +O_\eps(\eps^2)$, it suffices to prove the same bounds for $(\mathcal{L}_{\mu, \eps}W_j, U_k)$. In view of \eqref{decompose:L} we write 
\[
(\mathcal{L}_{\mu, \eps}W_j, U_k)=(\mathcal{L}_{0, \eps}W_j, U_k)+\mu ((L^1_\eps+L^\sharp) W_j, U_k).
\]
By \eqref{est:W} we have $\mu ((L^1_\eps+L^\sharp) W_j, U_k)=O(\mu^2)$, 
so that it remains to consider $(\mathcal{L}_{0, \eps}W_j, U_k)$.   
From the Neumann series \eqref{Neumann:Wj} we have 
\[
W_j=-\mu\Xi_\eps\Pi (L^1_\eps+L^\sharp) U_j+O(\mu^2+|\ld|^2). 
\]
Hence
\bq \label{Neu} 
(\mathcal{L}_{\mu, \eps}W_j, U_k)=-\mu \big(\mathcal{L}_{0, \eps}\Xi_\eps\Pi L^1_\eps  U_j, U_k\big)
- \mu \big(\mathcal{L}_{0, \eps}\Xi_\eps\Pi L^\sharp U_j, U_k\big)+O(\mu^2+|\ld|^2),
\eq
where  $L^\sharp U_j=0$ for $j\in \{2, 3, 4\}$.
We recall  that $U_1$ and $U_2$ are $(\mathrm{odd}, \mathrm{even})^T$, 
whereas $U_3$ and $U_4$ are $(\mathrm{even}, \mathrm{odd})^T$.
 By Lemma \ref{lemm:parity}, $\Pi$ preserves the parity, while $\Xi_\eps$ switches the parity. On the other hand, it is  easy to check that $L^1_\eps$  preserves the parity, while $\mathcal{L}_{0, \eps}$ 
 switches the parity. Consequently, $\mathcal{L}_{0, \eps}\Xi_\eps\Pi L^1_\eps$ preserves the parity. We deduce that if $U_j$ and $U_k$ have opposite parity, 
 then so do $\mathcal{L}_{0, \eps}\Xi_\eps\Pi L^1_\eps U_j$ and $U_k$.  This observation implies that 
 \bq\label{L0:cross}
 (\mathcal{L}_{0, \eps}\Xi_\eps\Pi L^1_\eps U_j, U_k)=0
 \eq
 both for $j\in \{1, 2\}$, $k\in \{3, 4\}$ and for $j\in \{3, 4\}$, $k\in \{1, 2\}$.   
 Thus the first term in \eqref{Neu} also vanishes, so we obtain  
 \bq\label{estB:1:0}
 (\mathcal{L}_{\mu, \eps}W_j, U_k)=O(\mu^2+|\ld|^2)
 \eq
 both for $j\in \{3, 4\},~k\in\{1, 2\}$ and for $j=2$, $k\in \{3, 4\}$. In particular, this proves \eqref{estB:lowerleft}. 
 
 In order to prove the other estimates, we use $\mathcal{L}_{0, \eps}^*=J\mathcal{L}_{0, \eps}J$ (see \eqref{L*}) to have
 \bq\label{estB:100}
 \begin{aligned}
 (\mathcal{L}_{\mu, \eps}W_j, U_k)&= -\mu\big(\mathcal{L}_{0, \eps}\Xi_\eps\Pi (L^1+L^\sharp)  U_j, U_k\big)+O(\mu^2+|\ld|^2)\\
 &=-\mu\big(\Xi_\eps\Pi (L^1+L^\sharp)  U_j, \Pi J \mathcal{L}_{0, \eps}J U_k\big)+O(\mu^2+|\ld|^2). 
 \end{aligned}
 \eq
According to \eqref{expand:Pi} and \eqref{L1U1},
 \begin{align*}
 &\Pi L^\sharp U_1= \Pi \begin{bmatrix} 1\\ 0\end{bmatrix}=\begin{bmatrix} 1\\ 0\end{bmatrix}-\frac{(L^\sharp U_1, U_4)}{(U_4, U_4)}U_4+O_\eps(\eps^2)= O_\eps(\eps)\\
  &\Pi L^1 U_1= \Pi \begin{bmatrix} 0\\ i\end{bmatrix}+O_\eps(\eps)=O_\eps(\eps).
 \end{align*}
  It follows from this, \eqref{estB:100} and \eqref{boundJLJ} that
   \[
(\mathcal{L}_{\mu, \eps}W_1, U_k)=\mu O_\eps(\eps^2)+O(\mu^2+|\ld|^2)\quad \forall k,
\]
which finishes the proof of \eqref{estB:1k}.  
The proof of \eqref{estB:4k} is similar to \eqref{estB:1k} since $L^\sharp U_4=0$ and 
\[
\Pi L^1 U_4=\Pi \begin{bmatrix} i\\0 \end{bmatrix}+O_\eps(\eps)=O_\eps(\eps) 
\]
 by \eqref{L1U4}. Next, it can be directly checked that
\begin{align}
&\Pi L^1U_2= \begin{bmatrix} \frac{i}{2}S\\ \frac{i}{2}C\end{bmatrix}+\eps  \begin{bmatrix}iS_2\\ -\frac{i}{2}C_2\end{bmatrix}+O_\eps(\eps^2),\\\label{XiPiL1U2:0}
&\Xi_\eps \Pi L^1U_2=\frac{i}{4} \begin{bmatrix}-C\\S\end{bmatrix}+\frac{i}{4}\eps  \begin{bmatrix}1-6C_2\\-3S_2\end{bmatrix}+O_\eps(\eps^2),
\end{align}
 where $L^1U_2$ is given by \eqref{L1U2}. We note that \eqref{XiPiL1U2:0} can be checked by applying the operator
$\Pi\cL_{0,\eps}$ to the right side of \eqref{XiPiL1U2:0}. Taking the inner product with \eqref{JLJU1} and \eqref{JLJU2} gives 
\[
 (\mathcal{L}_{\mu, \eps}W_2, U_k)=
 \begin{cases}
\tdm \pi i \mu \eps+\mu O_\eps(\eps^2)+O(\mu^2+|\ld|^2),\quad k=1,\\
 \mu O_\eps(\eps^2)+O(\mu^2+|\ld|^2),\quad k=2,
 \end{cases}
\]
which  yields \eqref{estB:21} and \eqref{estB:22}. Similarly, we have 
\begin{align}
\Pi L^1U_3=\frac{i}{2}\begin{bmatrix}-C \\S \end{bmatrix}+\eps\frac{i}{2}\begin{bmatrix}1-2C_2\\-S_2 \end{bmatrix}+O_\eps(\eps^2),\\
\Xi_\eps\Pi L^1U_3=-\frac{i}{4}\begin{bmatrix}S \\ C \end{bmatrix}+\eps\frac{3i}{4}\begin{bmatrix}-2S_2\\C_2 \end{bmatrix}+O_\eps(\eps^2).
\end{align}
Consequently, we obtain in view of \eqref{estB:100}, \eqref{JLJU3} and \eqref{JLJU4} that 
\[
 (\mathcal{L}_{\mu, \eps}W_3, U_k)=
 \begin{cases}
\mu O_\eps(\eps^2)+O(\mu^2+|\ld|^2),\quad k=3,\\
i\pi \mu \eps +\mu O_\eps(\eps^2)+O(\mu^2+|\ld|^2),\quad k=4,
 \end{cases}
\]
whence \eqref{estB:33} and \eqref{estB:34} follow. 

Finally, let us prove \eqref{estB:23} and \eqref{estB:24}, which   are an improvement of  \eqref{estB:1:0} for $j=2$. Indeed, using \eqref{Neumann:Wj} and \eqref{decompose:L} we obtain 
\bq\label{form:LW2Uk}
\begin{aligned}
(\mathcal{L}_{\mu, \eps}W_2, U_k)& = -\mu \big(\mathcal{L}_{0, \eps}\Xi_\eps\Pi L^1_\eps U_2, U_k\big) 
+\mu \big(\mathcal{L}_{0, \eps}\Xi_\eps \Pi [\mu(L^1_\eps 
+L^\sharp)-\ld \mathrm{Id}]\Xi_\eps\Pi L^1_\eps U_2, U_k\big)\\
&\quad -\mu^2\big((L^1_\eps+L^\sharp)\Xi_\eps\Pi L^1_\eps U_2, U_k\big)+O(\mu^3+|\ld|^3)\\
&=:I+II+III+O(\mu^3+|\ld|^3),
\end{aligned}
\eq
 where  $k\in \{3, 4\}$. We recall from \eqref{L0:cross}  that $I=0$.  
Next we write the second term as 
\begin{align*}
II= \mu \big(\Xi_\eps \Pi [\mu(L^1_\eps +L^\sharp)-\ld \mathrm{Id} 
]\Xi_\eps\Pi L^1_\eps U_2, \Pi J\mathcal{L}_{0, \eps}JU_k\big)
\end{align*}
and recall \eqref{boundJLJ} and \eqref{def:Pi34} to have 
\bq\label{est:II:B2k}
II=\eps \Pi_2(\mu ,\ld).
\eq 
As for $III$, 
we  compute  
\bq\label{L1XiPiL1U2}
  L^1_\eps\Xi_\eps \Pi L^1_\eps U_2 
  = -\frac{1}{4}\begin{bmatrix}-2C\\ S \end{bmatrix} 
  -\frac{\eps}{4}\begin{bmatrix}2-2C_2\\-4S_2\end{bmatrix}+O_\eps (\eps^2)  \quad 
  \text{ and  } \quad  L^\sharp\Xi \Pi L^1_\eps U_2=O_\eps (\eps^2) 
\eq
using \eqref{XiPiL1U2:0}.   Consequently,
\bq\label{est:III:B2k}
 III=
 \begin{cases}
 -\frac{\pi}{4}\mu^2+\mu^2O_\eps(\eps^2),\quad k=3,\\
\frac{\pi}{4}\mu^2\eps+\mu^2O_\eps(\eps^2),\quad k=4,
 \end{cases}
 \eq
 which combined with \eqref{est:II:B2k} completes the proof of \eqref{estB:23} and \eqref{estB:24}.
\end{proof}

\section{Proof of the modulational instability}
By virtue of Proposition \ref{prop:instab}, the proof of modulational instability reduces to proving the existence of a small root $\ld$ of \eqref{def:cP} with positive real part.
\subsection{Expansion of $\cP(\ld; \mu, \eps)$}
We determine the contribution of ${\bf B}_{\mu, \eps}$ in $\cP(\ld; \mu, \eps)=\det({\bf A}_{\mu, \eps}-\ld {\bf I}_\eps+{\bf B}_{\mu, \eps})$ by inspecting the individual terms of the determinant.  The terms that involve $\bf B_{\mu,\eps}$ are estimated as follows.  
\begin{prop}\label{prop:detB}
The sideband terms in $\cP$ are
\bq\label{contribution:B}
\begin{aligned}
\cP(\ld; \mu, \eps)-\det({\bf A}_{\mu, \eps}-\ld {\bf I}_\eps)&=-\frac{1}{8}\mu^3\eps^2+\mu (\mez i\mu -\ld)^2 O_\eps(\eps^2)+\mu^2(\mez i\mu -\ld)O_\eps(\eps^2) \\
&\quad +O_\eps(\eps^3)\Pi_3(\mu, \ld)+O(\mu^4+|\ld|^4),
\end{aligned}
\eq
where we recall that  $\Pi_3(\mu, \ld)$ denotes any polynomial of the form $a_1\mu^3+a_2\mu^2\ld+a_3\mu\ld^2$.
\end{prop}
\begin{rema}
Analogously to \eqref{formdetA}, we observe that both of the $O_\eps(\eps^2)$ terms in \eqref{contribution:B} have the factor $(\mez i\mu -\ld)$.
\end{rema}
\begin{proof}[Proof of Proposition \ref{prop:detB}] 
For notational simplicity, we write ${\bf A}_{\mu, \eps}={\bf A}$, ${\bf B}_{\mu, \eps}={\bf B}$ and ${\bf I}_\eps={\bf I}$.  
We shall treat any term that is either $O_\eps(\eps^3)\Pi_3(\mu, \ld)$ or  $O(\mu^4+|\ld|^4)$ as a remainder.  
Let us break $4\times 4$ matrices into four $2\times 2$ blocks.  
We  observe that in ${\bf A}-\ld {\bf I}$, given by \eqref{final:A-I}, the only entries without $\mu$ or $\ld$ 
are the $(3, 1)$, $(3, 2)$ and $(4, 1)$ entries, all of which are in the lower left block. In addition, ${\bf B}=O(\mu)$. 
Thus,  possibly except for terms containing entries from the lower left block, each term in the 
Leibniz formula for $\det({\bf A}-\ld {\bf I}+{\bf B})$ and for $\det({\bf A}-\ld {\bf I})$ is  $O(\mu^4+|\ld|^4)$.  
We are left with two types of terms: terms containing exactly one entry, which we call type $I$ terms,  
and those containing two entries of the lower left block, which we call type $II$ terms.   

Among terms of  type $I$, if the only entry of the lower left block comes from ${\bf B}$, 
then it is $O(\mu^4+|\ld|^4)$ thanks to \eqref{estB:lowerleft}.   
It thus suffices to consider type $I$ terms  that  have exactly one entry of ${\bf A}-\ld\bf{I}={\bf A}$ from the lower left  block.  
Noting in addition that ${\bf A}_{31}=O(\eps^4)$, we deduce that the contribution of type $I$ is
 \bq\label{detA-detB:2}
\begin{aligned}
I&= -{\bf A}_{32}({\bf A}_{11}-\ld+{\bf B}_{11}){\bf B}_{23}({\bf A}_{44} 
+{\bf B}_{44})+ {\bf A}_{32}({\bf A}_{11}-\ld+{\bf B}_{11}){\bf B}_{24}({\bf A}_{43}-\ld{\bf I}_{43}+{\bf B}_{43})\\
&\quad +{\bf A}_{32}({\bf A}_{21}+{\bf B}_{21})({\bf A}_{13}+{\bf B}_{13})({\bf A}_{44}-\ld +{\bf B}_{44})\\
& \quad-{\bf A}_{32}({\bf A}_{21}+{\bf B}_{21})({\bf A}_{43}-\ld{\bf I}_{43}+{\bf B}_{43})({\bf A}_{14}+{\bf B}_{14})\\
&\quad-{\bf A}_{41}({\bf A}_{12}+{\bf B}_{12}){\bf B}_{23}({\bf A}_{34}-\ld{\bf I}_{34}+{\bf B}_{34}) 
+{\bf A}_{41}({\bf A}_{12}+{\bf B}_{12})({\bf A}_{33}-\ld +{\bf B}_{33}){\bf B}_{24}\\
&\quad+{\bf A}_{41}({\bf A}_{22}-\ld +{\bf B}_{22})({\bf A}_{13}+{\bf B}_{13})({\bf A}_{34}-\ld{\bf I}_{34}+{\bf B}_{34})\\
&\quad-{\bf A}_{41}({\bf A}_{22}-\ld +{\bf B}_{22})({\bf A}_{33}-\ld +{\bf B}_{33})({\bf A}_{14}+{\bf B}_{14})\\
&\quad+O(\eps^4)\Pi_3(\mu, \ld)+O(\mu^4+|\ld|^4)\\
&=\sum_{m=1}^8I_m+O(\eps^4)\Pi_3(\mu, \ld)+O(\mu^4+|\ld|^4).
\end{aligned}
\eq
By  \eqref{estB:23} and \eqref{estB:24} we have ${\bf B}_{23},~ {\bf B}_{24}=O(\mu^2+|\ld|^2)$, so that 
\bq\label{Im}
I_m=O(\mu^4+|\ld|^4),\quad m\in \{1, 2, 5, 6 \}.
\eq
 Using Lemma \ref{lemm:B:1} we find that 
\bq\label{I347}\begin{aligned}
&I_3={\bf A}_{32}{\bf A}_{21}{\bf A}_{13}({\bf A}_{44}-\ld)+O_\eps(\eps^5)\Pi_3(\mu, \ld)+O(\mu^4+|\ld|^4),\\
&I_4=-{\bf A}_{32}{\bf A}_{21}({\bf A}_{43}-\ld{\bf I}_{43}){\bf A}_{14}+ O_\eps(\eps^4)\Pi_3(\mu, \ld)+O(\mu^4+|\ld|^4),\\
&I_7={\bf A}_{41}({\bf A}_{22}-\ld ){\bf A}_{13}({\bf A}_{34}-\ld{\bf I}_{34})+ O_\eps(\eps^3)\Pi_3(\mu, \ld)+O(\mu^4+|\ld|^4).
\end{aligned}
\eq
Next we expand $I_8$ as
\begin{align*}
I_8&=-{\bf A}_{41}({\bf A}_{22}-\ld )({\bf A}_{33}-\ld){\bf A}_{14}-{\bf A}_{41}{\bf B}_{22}({\bf A}_{33}-\ld ){\bf A}_{14} 
-{\bf A}_{41}({\bf A}_{22}-\ld){\bf B}_{33}{\bf A}_{14}\\
&\quad -{\bf A}_{41}({\bf A}_{22}-\ld )({\bf A}_{33}-\ld){\bf B}_{14}-{\bf A}_{41}{\bf B}_{22}{\bf B}_{33}{\bf A}_{14} 
-{\bf A}_{41}({\bf A}_{22}-\ld){\bf B}_{33}{\bf B}_{14}\\
&\quad-{\bf A}_{41}{\bf B}_{22}({\bf A}_{33}-\ld){\bf B}_{14}-{\bf A}_{41}{\bf B}_{22}{\bf B}_{33}{\bf B}_{14} 
=I_{8, 0}+I_{8, 1}+\dots I_{8,7}.
\end{align*}
By virtue of Lemma \ref{lemm:B:1} we have 
\begin{align*}
&I_{8, m}=\mu^2(\mez i\mu -\ld)O_\eps(\eps^2)+O_\eps(\eps^4)\Pi_3(\mu, \ld)+O(\mu^4+|\ld|^4),\quad m=1,2,\\
&I_{8, 3}=\mu (\mez i\mu -\ld)^2 O_\eps(\eps^2)+O_\eps(\eps^4)\Pi_3(\mu, \ld)+O(\mu^4+|\ld|^4),\\
&I_{8, m}=O_\eps(\eps^4)\Pi_3(\mu, \ld)+O(\mu^4+|\ld|^4),\quad m=4, 5, 6, 7.
\end{align*}
 Gathering the preceding estimates yields 
\bq\label{I8}
\begin{aligned}
I_8&=-{\bf A}_{41}({\bf A}_{22}-\ld )({\bf A}_{33}-\ld){\bf A}_{14}+\mu (\mez i\mu -\ld)^2 O_\eps(\eps^2) 
+\mu^2(\mez i\mu -\ld)O_\eps(\eps^2)\\
&\quad+O_\eps(\eps^4)\Pi_3(\mu, \ld)+O(\mu^4+|\ld|^4).
\end{aligned}
\eq
Combining \eqref{Im}, \eqref{I347} and \eqref{I8}, 
we deduce that the total contribution of ${\bf B}$ in type $I$ terms of $\det({\bf A}-\ld {\bf I}+{\bf B})$ is 
\bq\label{IB}
\begin{aligned}
I^{{\bf B}}&=\mu (\mez i\mu -\ld)^2 O_\eps(\eps^2)+\mu^2(\mez i\mu -\ld)O_\eps(\eps^2) 
+O_\eps(\eps^3)\Pi_3(\mu, \ld)+O(\mu^4+|\ld|^4).
\end{aligned}
\eq

The contribution of the type $II$ terms is
\bq\label{detA-detB:1}
\begin{aligned}
II&:=({\bf A}_{31} +{\bf B}_{31}){\bf B}_{42}({\bf A}_{13} +{\bf B}_{13}){\bf B}_{24} 
-({\bf A}_{31} +{\bf B}_{31}){\bf B}_{42}({\bf A}_{14} +{\bf B}_{14}){\bf B}_{23}\\
&-({\bf A}_{41} +{\bf B}_{41})({\bf A}_{32} +{\bf B}_{32})({\bf A}_{13} +{\bf B}_{13}){\bf B}_{24} 
+({\bf A}_{41} +{\bf B}_{41})({\bf A}_{32} +{\bf B}_{32})({\bf A}_{14} +{\bf B}_{14}){\bf B}_{23},
\end{aligned}
\eq
where we have used the facts that ${\bf A}_{23}={\bf A}_{24}={\bf A}_{42}=0$ and $\bf{I}=0$ 
in the lower left and upper right blocks. Notice that each term in $II$ contains at least one entry of ${\bf B}$.  
 In the process of expanding each product in \eqref{detA-detB:1}, if there are at least three entries of ${\bf B}$, 
then at least one of the three comes from the lower left block of ${\bf B}$.  
 So this one is $O(\mu^2+|\ld|^2)$ 
by virtue of \eqref{estB:lowerleft}, implying that the term is $O(\mu^4+|\ld|^4)$.  Therefore 
 we are left with
\bq\label{expand:II}
\begin{aligned}
II&={\bf A}_{31} {\bf B}_{42}{\bf A}_{13} {\bf B}_{24}-{\bf A}_{31} {\bf B}_{42}{\bf A}_{14} {\bf B}_{23}
-{\bf B}_{41}{\bf A}_{32}{\bf A}_{13}{\bf B}_{24}-{\bf A}_{41} {\bf B}_{32}{\bf A}_{13} {\bf B}_{24}\\
&\quad-{\bf A}_{41} {\bf A}_{32} {\bf B}_{13}{\bf B}_{24}-{\bf A}_{41} {\bf A}_{32} {\bf A}_{13} {\bf B}_{24}
+{\bf B}_{41}{\bf A}_{32} {\bf A}_{14} {\bf B}_{23}+{\bf A}_{41} {\bf B}_{32}{\bf A}_{14} {\bf B}_{23}\\
&\quad+{\bf A}_{41} {\bf A}_{32} {\bf B}_{14}{\bf B}_{23}+{\bf A}_{41} {\bf A}_{32} {\bf A}_{14}{\bf B}_{23}
+O(\mu^4+|\ld|^4)\\
&=:\sum_{m=1}^{10} II_m+O(\mu^4+|\ld|^4).
\end{aligned}
\eq
Within $II_m$ for $m\in\{ 1, 2, 3, 4, 7, 8\}$, there is one entry from the lower left block of  ${\bf B}$, 
one entry from the upper right  block of  ${\bf B}$ and one entry from the upper right block of ${\bf A}$. 
So their product is $O(\mu^4+|\ld|^4)$ in view of \eqref{estB:lowerleft} 
and the fact that ${\bf A}=O(\mu)$ in the upper right block. 
On the other hand, from \eqref{estB:1k}, \eqref{estB:23}  and \eqref{estB:24} we find that 
\begin{align*}
II_5&=-\eps^2 \big[ \mu O_\eps (\eps^2)+O(\mu^2+|\ld|^2)\big]\big[\eps \Pi_2(\mu, \ld)+O(\mu^3+|\ld|^3)\big]\\
& =O_\eps(\eps^5)\Pi_3(\mu, \ld)+O(\mu^4+|\ld|^4),\\
II_6&=-\eps^2 \big[\mu O_\eps(\eps^2)\big]\big[\eps \Pi_2(\mu, \ld)+O(\mu^3+|\ld|^3)\big]\\
& =O_\eps(\eps^5)\Pi_3(\mu, \ld)+O(\mu^4+|\ld|^4),\\
II_9&=\eps^2\big[ \mu O_\eps(\eps^2)+O(\mu^2+|\ld|^2)\big] 
\big[ -\frac{1}{8}\mu^2+\eps \Pi_2(\mu, \ld)+O(\mu^3+|\ld|^3)\big]\\
&=O_\eps(\eps^4)\Pi_3(\mu, \ld)+O(\mu^4+|\ld|^4),\\
II_{10}&=\eps^2\big[\mu+\mu O_\eps(\eps^2)\big]\big[ -\frac{1}{8}\mu^2+\eps \Pi_2(\mu, \ld)+O(\mu^3+|\ld|^3)\big]\\
&=-\frac{1}{8}\mu^3\eps^2+O_\eps(\eps^3)\Pi_3(\mu, \ld)+O(\mu^4+|\ld|^4).
\end{align*}
Thus the  total contribution of ${\bf B}$ in type $II$ terms is
\bq\label{IIB}
II^{\bf B}=-\frac{1}{8}\mu^3\eps^2+O_\eps(\eps^3)\Pi_3(\mu, \ld)+O(\mu^4+|\ld|^4).
\eq
Finally, combining \eqref{IB} and \eqref{IIB} leads to \eqref{contribution:B}.
 \end{proof}
 Now combining Propositions \ref{prop:detA} and \ref{prop:detB} we obtain the expansion for $\cP$
 \bq
 \begin{aligned}
\cP(\ld; \mu, \eps) = &
\big(\mez i\mu-\ld \big)^2\mu+\eps^2\Big\{-\frac{1}{8}\mu^3+r_1\mu\big(\mez i\mu -\ld\big)^2 
+r_2\mu^2\big(\mez i\mu -\ld\big)\Big\}\\  
&+O_\eps(\eps^3)\Pi_3(\mu, \ld) +O(\mu^4+|\ld|^4).
\end{aligned}
\eq
for some absolute constants $r_1,~r_2\in \Cc$.   
Still for small $\mu \in (0, \mez)$, we set 
\bq\label{def:gamma}
\ld =\gamma \mu
\eq
so that, upon recalling  \eqref{def:Pi34}, we have 
\bq\label{def:wtP}
\cP(\ld; \mu, \eps)=\mu^3 \wt P(\gamma; \mu, \eps),
\eq
where
\bq\label{form:wtP}
\begin{aligned}
&\wt P(\gamma; \mu, \eps)=  \big(\mez i-\gamma \big)^2+\eps^2 \Big\{-\frac{1}{8}+r_1 \big(\mez i -\gamma \big)^2  
+r_2 \big(\mez i -\gamma \big)\Big\}+O_\eps(\eps^3)\tt_1(\gamma)+\mu \tt_2(\gamma; \mu, \eps)
\end{aligned}
\eq
for some smooth function  $\tt_2(\gamma; \mu, \eps)$ and  for some quadratic $\tt_1(\gamma)$.   
The principal part of $\wt \cP$ with the last term omitted  is  
\bq\label{def:Q}
\wt Q(\gamma;  \eps)=  \big(\mez i-\gamma \big)^2+\eps^2 \Big\{-\frac{1}{8}+r_1 \big(\mez i -\gamma \big)^2  
+r_2 \big(\mez i -\gamma \big)\Big\}+O_\eps(\eps^3)\tt_1(\gamma).
\eq
Clearly, $\wt Q$ is a quadratic polynomial in $\gamma$.

\subsection{Roots of the characteristic function $\cP(\gamma; \mu,  \eps)$}
 First we look for the roots of $\wt Q$. 
 Of course, for  $\eps=0$,  $\wt Q(\gamma;  0)=(\mez i-\gamma)^2$ has the imaginary double  root $\mez i$.  
 We will prove that for small $\eps \ne 0$, the double root $\mez i$ bifurcates off the imaginary axis, which will subsequently lead to an unstable eigenvalue of $\mathcal{L}_{\mu, \eps}$.
\begin{lemm}\label{prop:Q}
There exists a small $ \eps_0>0$ such that for all $\eps\in  (-\eps_0, \eps_0)\setminus \{0\}$,  the quadratic polynomial  $\wt Q(\gamma;  \eps)$ has two simple roots
\bq\label{gamma:Q}
\gamma_\pm(\eps)=\mez i+\eps \ka_\pm(\eps),
\eq
where $\ka_\pm:  (-\eps_0, \eps_0)\to \Rr$  are smooth functions  and $\ka_\pm(0)=\pm\frac{1}{2\sqrt{2}}$. 
\end{lemm}
\begin{proof}
We seek solutions of the form $\gamma=\mez i+\ka \eps$. Then from \eqref{def:Q} we have
\[
\wt Q\big(\mez i+\ka\eps;  \eps\big)=\eps^2(\ka^2-\frac{1}{8})+\eps^2(r_1\ka^2\eps^2 -r_2\ka\eps)+O_\eps(\eps^3)\tt_1\big(\mez i+\ka \eps\big).
\]
Recall that $O_\eps(\eps^3)$ depends only on $\eps$. Dividing through by $\eps^2\ne 0$, we see that  $\wt Q(\mez i+\ka \eps;  \eps)$ has the same roots $\ka$ as $ Q^\sharp(\ka; \eps)$ where
\bq\label{def:Qsharp}
\begin{aligned}
Q^\sharp(\ka;  \eps)&:=\eps^{-2}\wt Q\big(\mez i+\ka\eps^2; \eps\big)=(\ka^2-\frac{1}{8})+\eps(r_1\ka^2\eps -r_2\ka)+O_\eps(\eps)\tt_1\big(\mez i+\ka \eps\big).
\end{aligned}
\eq
Clearly, $\ka^0_\pm=\pm\frac{1}{2\sqrt{2}}$ are the roots of $Q^\sharp(\cdot; 0)$.  
Since $\p_\ka Q^\sharp(\ka^0_\pm; 0)=\pm\frac{1}{\sqrt{2}}\ne 0$, the Implicit Function Theorem implies that 
there exists a pair of smooth functions $\ka_\pm(\eps)$ such that $\ka_\pm(0)=\ka_\pm^0$ 
and $Q^\sharp(\ka_\pm(\eps);  \eps)=0$ for small $\eps$.  
From the definition $\wt Q(\mez i+\ka \eps ;  \eps)=\eps^2 Q^\sharp(\ka; \eps)$, 
the roots of $\wt Q(\gamma; \eps)$ for small $\eps$ are $\gamma_\pm(\eps)=\mez i+\eps \ka_\pm(\eps)$.   
Since 
\[
\p_\gamma\wt Q \big(\mez i+\ka_\pm^0\eps;  \eps\big) = \eps \p_\ka Q^\sharp(\ka_\pm^0; \eps)=\pm\frac{\eps}{\sqrt{2}},
\]
we have $\p_\gamma\wt Q(\gamma_\pm(\eps);  \eps)\ne 0$ for small $\eps \ne 0$, implying that $\gamma_\pm(\eps)$ are simple roots for small $\eps \ne 0$.
\end{proof}
Now we recall from \eqref{form:wtP} and \eqref{def:Q} that 
 \bq\label{wtP-Q}
\wt \cP(\gamma; \mu, \eps)=\wt Q(\gamma; \eps)+\mu \tt_2(\gamma; \mu, \eps).
 \eq
 In particular, $\wt \cP(\gamma; 0, \eps)=\wt Q(\gamma; \eps)$.  
 If we fix a small $ \eps\in (-\eps_0, \eps_0)\setminus\{0\}$ and vary $\mu$, according to Lemma \ref{prop:Q}, the polynomial  $\wt Q(\gamma; \eps)$ 
 has two simple  roots of the form  $\gamma_\pm(\eps)= \mez i+\eps \ka_\pm(\eps)$.  In particular, 
  \[
  \p_\gamma\wt\cP(\gamma; 0, \eps)\vert_{\gamma= \gamma_\pm(\eps)}=\p_\gamma \wt Q(\gamma; \eps)\vert_{\gamma= \gamma_\pm(\eps)}\ne 0.
  \]
The Implicit Function Theorem  applied to \eqref{wtP-Q} implies that 
for each $\eps\in (-\eps_0, \eps_0)\setminus\{0\}$ there exists a small $\mu_0(\eps)>0$ such that for all $\mu\in (0,  \mu_0(\eps))$,  
$\wt \cP(\gamma; \mu, \eps)$ has at least two simple roots $ \gamma_\pm(\mu, \eps)$.  
For each such $\eps$, both mappings $\mu\mapsto \gamma_\pm(\mu ,\eps)$ are smooth and
 \bq\label{gamma:mu=0}
 \gamma_\pm(0, \eps)=\mez i+\eps \ka_\pm(\eps),\quad \ka_\pm(0)=\pm\frac{1}{2\sqrt{2}}.
\eq
 Finally, recalling the scaling relations \eqref{def:gamma}, \eqref{def:wtP} and \eqref{spec:negmu} 
 we obtain our main conclusion,  as follows.
 \begin{theo}\label{prop:P}
For all $\eps\in (-\eps_0, \eps_0)\setminus\{0\}$ and  $\mu\in \big(0, \mu_0(\eps)\big)$,  $\cP(\ld; \mu, \eps)$ 
has at least two simple roots of the form 
\bq
\ld_\pm(\mu, \eps)= \mu \gamma_\pm(\mu, \eps),
\eq
 where $\mu\mapsto \gamma_\pm(\mu, \eps)$ are smooth and satisfy \eqref{gamma:mu=0}. In particular,  
 \bq
\ld_\pm(\mu, \eps)=\mez i\mu\pm \frac{1}{2\sqrt{2}}\mu\eps+\mu\eps^2g_1(\eps)+\mu^2g_2(\mu, \eps),
\eq
where $g_1(\cdot)$ and  $g_2(\cdot ,\eps)$ are smooth for each $\eps$. On the other hand, for  $\mu \in \big(-\mu_0(\eps), 0\big)$ we have
 \bq
\ld_\pm(\mu, \eps)=\mez i\mu\mp \frac{1}{2\sqrt{2}}\mu\eps-\mu\eps^2\overline{g_1}(\eps)+\mu^2\overline{g_2}(-\mu, \eps).
\eq
 \end{theo}
 Theorem \ref{prop:P} completes the proof of the modulational instability for Stokes waves of small amplitude in deep water. 
 
\appendix
\section{Stokes wave expansion}\label{appendix:Stokes}
  Here we derive from scratch the expansion of a Stokes wave of small amplitude and zero Bernoulli constant, $P=0$.  
 Our motivation is that the expansions found in the literature seem to be not unique.  
 In fact, the apparent non-uniqueness is simply due to different choices of coordinates for the parameter $a$. 
 
 In the moving frame of speed $c$, 
 the water wave system \eqref{ww:0}  becomes
  \begin{align}\label{ww:eq1}
& \Delta_{x, y} \phi =0\quad \text{in } \Omega, \\ \label{ww:eq2}
&-c\p_x\phi + g\eta + \tfrac12 |\na_{x, y}\phi|^2 =0\quad \text{ on } \{y=\eta(x)\}, \\ \label{ww:eq3}
 &\p_y\phi + (c-\p_x\phi)\p_x\eta=0\quad  \text{ on } \{y=\eta(x)\}, \\ \label{ww:eq4}
 &\na_{x, y}\phi \to 0 \text{ as } y\to -\infty. 
\end{align}
Using superscripts we Taylor-expand the unknowns,  
\begin{align*}
&\eta = \eps\eta^1 + \eps^2\eta^2 +\eps^3\eta^3+ \dots,\\
&\phi = \eps\phi^1 + \eps^2\phi^2 + \eps^3\phi^3+\dots,\\
&c = c^0+\eps c^1 +\eps^2c^2  +\dots
\end{align*}
  and reserve subscripts for derivatives. Each $\phi^j$ is harmonic in $\{ y<0\}$.   Then we Taylor-expand 
  \begin{align} 
\phi(x,\eta(x)) = &\phi(x, 0) + \phi_y(x,0)[\eps\eta^1(x) + \eps^2\eta^2(x) + \eps^3\eta^3(x) +\dots] \\
&+ \tfrac12\phi_{yy}(x,0)[\eps\eta^1(x)+\dots]^2 +\dots, 
\end{align} 
and similarly for $(\p_x\phi)(x, \eta(x))$ and $(\p_y\phi)(x, \eta(x))$.  In the following we will suppress the arguments.  In most places the arguments of $\phi, \phi_x, \phi_y$, etc. will be $(x,0)$. Equation \eqref{ww:eq3} gives
\bq\label{expand:wweq3}
\begin{aligned} 
&\eps\{\phi^1_y+ c^0\eta_x^1\}  
+  \eps^2\Big\{\phi_y^2 + \phi_{yy}^1\eta^1 - \phi_x^1\eta_x^1 + c^0\eta_x^2 + c^1\eta_x^1\Big\} \\
&+\eps^3\Big\{ \phi_y^3 + \phi_{yy}^1\eta^2 + \phi_{yy}^2\eta^1  - \phi_x^1\eta_x^2 - \phi_x^2\eta_x^1 -\phi_{xy}^1\eta^1\eta_x^1  \\
&+ c^0\eta_x^3  + c^1\eta_x^2 + c^2\eta_x^1 +\tfrac12\phi_{yyy}^1\eta^1\eta^1 \Big\} + O(\eps^4) = 0.  
\end{aligned}
\eq 
On the other hand, equation \eqref{ww:eq2} gives
\bq\label{expand:wweq2}
\begin{aligned}
&\eps\{ -c^0\phi_x^1 + g\eta^1\}  
+  \eps^2\Big\{-c^0\phi_x^2 - c^0\phi_{xy}^1\eta^1 - c^1\phi_x^1 + g\eta^2 + \tfrac12[\phi_x^1]^2 + \tfrac12[\phi_y^1]^2  \Big\} \\
&+ \eps^3\Big\{ -c^0\phi_x^3 - c^0\phi_{xy}^2\eta^1 - c^0\phi_{xy}^1\eta^2 - c^1\phi_x^2 -c^2\phi_x^1  \\
&+ g\eta^3 + \phi_x^1\phi_{xy}^1\eta^1 + \phi_y^1\phi_{yy}^1\eta^1 - \tfrac{c^0}{2}\phi_{xyy}^1\eta^1\eta^1-c^1\phi^1_{xy}\eta^1+\phi_x^1\phi_x^2+\phi_y^1\phi_y^2  \Big\}+ O(\eta^4) = 0. 
\end{aligned}
\eq
Now equating the coefficients of $\eps$ yields
\bq  \phi_y^1(x,0) + c^0\eta_x^1(x) = 0, \quad -c^0\phi_x^1(x,0) + g\eta^1(x) = 0, \quad \phi^1_{xx} + \phi^1_{yy}=0.  \eq
Clearly a solution is 
\bq\label{Stokes:order1}     \eta^1(x)=\cos x, \quad \phi^1(x,y)=c^0e^y\sin x, \quad c^0=\sqrt{g}. 
\eq  
In the coefficients of $\eps^2$, we substitute \eqref{Stokes:order1} into \eqref{expand:wweq3} and \eqref{expand:wweq2} to obtain 
\[
\phi_y^2 + \sqrt g\,\eta_x^2 + (\sqrt g\sin x)(\cos x) - (\sqrt g\cos x)(-\sin x) + c^1(-\sin x) =0  
\]
and 
\[ 
-\sqrt g\phi_x^2 + g\eta^2 - \sqrt g(\sqrt g\cos x)(\cos x) - c^1(\sqrt g\cos x) + \tfrac12 (\sqrt g\cos x)^2 + \tfrac12(\sqrt g\sin x)^2 =0.  
\]
They simplify to 
\bq 
\phi_y^2 + \sqrt g\,\eta_x^2  - c^1\sin x +\sqrt g\sin(2x)=0, \quad 
-\phi_x^2 + \sqrt g\eta^2 - c^1\cos x -\tfrac12 \sqrt g\cos(2x) = 0.  \eq
We eliminate $\eta^2$ by combining these two equations as 
\bq
\phi_y^2(x, 0)+ \phi_{xx}^2(x, 0) - 2c^1\sin x = 0.  \eq
We choose the trivial solution 
\bq 
c^1=0,\quad  \phi^2=0, \quad \eta^2=\tfrac12 \cos(2x).  \eq

As for equating the coefficients of $\eps^3$, we may now put $\phi^2=0$ and $c^1=0$ to obtain from \eqref{expand:wweq3} and \eqref{expand:wweq2}   the equations 
\[ 
\phi_y^3 + \phi_{yy}^1\eta^2 - \phi^1_x\eta_x^2 - \phi_{xy}^1\eta^1\eta_x^1 
+ c^0\eta_x^3 + c^2\eta_x^1 + \tfrac12\phi_{yyy}^1\eta^1\eta^1  = 0 
\]
and 
\[
-c^0\phi_x^3 - c^0\phi_{xy}^1\eta^2 - c^2\phi_x^1 + g\eta^3 + \phi_x^1\phi_{xy}^1\eta^1 
+ \phi_y^1\phi_{yy}^1\eta^1 -\tfrac{c^0}{2}\phi_{xyy}^1\eta^1\eta^1  = 0.  
\]
Now we plug in $c^0=\sqrt g,\ \phi^1=\sqrt ge^y\sin x,\ \eta^1=\cos x,\ \eta^2=\tfrac12\cos(2x)$ to obtain 
\[
\phi_y^3 + \sqrt g\eta_x^3 - c^2\sin x 
+ \tfrac12\sqrt g\cos(2x)\sin x + \sqrt g\sin(2x)\cos x + \sqrt g\cos^2x\sin x +\tfrac12\sqrt{g} \sin x \cos^2x  =0\]
and 
\[
-\phi_x^3 + \sqrt g\eta^3 - c^2\cos x - \tfrac{\sqrt g}2 \cos(2x)\cos x + \sqrt g\cos^3x + \sqrt g\sin^2x\cos x - \tfrac12\sqrt g\cos^3x  = 0.  
\]
They simplify to 
\begin{align} 
\phi_y^3 + \sqrt g\eta_x^3 - c^2\sin x + \tfrac98\sqrt g \sin(3x) +\tfrac58\sqrt g\sin x = 0, \\ \label{eq:eta3}
-\phi_x^3 + \sqrt g\eta^3 - c^2\cos x + \sqrt g [\tfrac38\cos x - \tfrac38\cos(3x)]=0.  \end{align}
Combining the last two equations, we find 
\bq 
\phi_y^3(x, 0) + \phi_{xx}^3(x, 0) + (-2c^2 +\sqrt g)\sin x =0
\eq
which admits the (trival) solution 
\begin{equation}
c^2=\frac 12 \sqrt{g},\quad \phi^3=0.
\end{equation}
Then it follows from \eqref{eq:eta3} that
\begin{equation}
\eta^3=\frac{1}{8}\cos x+\frac{3}{8}\cos(3x).
\end{equation}
Thus we have proved the expansions for $\eta$ and $c$ in \eqref{expand:Stokes}. 
On the other hand,  since $\psi(x)=\phi(x, \eta(x))=\sqrt{g}e^{\eta(x)}\sin x$,  
the expansion for $\psi$ follows  from Taylor's formula. 
We remark that by a simple change of the variable $a$, we could  have modified  the coefficients of 
the $\frac18\cos x$ and $\frac{\sqrt{g}}{4}\sin x $ terms in \eqref{expand:star}  if we wished. 

\section{Riemann mapping and proof of Proposition \ref{prop:z12} and Lemma  \ref{lemm:Riemann}}\label{Appendix:Riemann}
\subsection{Riemann mapping}
Recall that the fluid domain at  a fixed time is given by $\Omega=\{(x, y)\in \Rr^2: y<\eta(x)\}$ where $\eta$ is $C^\infty$, even, $2\pi$-periodic, and $\eta(x)=O_\eps(\eps)$. We first prove the following Riemann mapping theorem  for the unbounded domain.
\begin{prop}\label{prop:Riemann}
For any sufficiently small $\eps$, there exist  mappings $Z_j(x, y): \Omega\to \Rr$, $j= 1, 2$ such that 
\begin{itemize}
\item[(i)] $Z_1+i Z_2$ is conformal in $\Omega$;
\item[(ii)] $(x, y)\in \Omega\mapsto (Z_1(x, y), Z_2(x, y))$ is one-to-one and onto $\Rr^2_-$;
\item[(iii)] $Z_2(x+2\pi , y)=Z_2(x, y)$  for all $(x, y)\in \Omega$ and $Z_2$ is even in $x$;  
\item[(iv)] $Z_2(x, \eta(x))=0$ for all $x\in \Rr$;
\item[(v)] $Z_1(x +2\pi, y)=2\pi +Z_1(x, y)$ for all $(x, y)\in \Omega$ and $Z_1$ is odd in $x$;
\item[(vi)] $\| \na_{x, y}(Z_1-x)\|_{L^\infty(\Omega)}+\| \na_{x, y}(Z_2-y)\|_{L^\infty(\Omega)}\le C\eps$.
\end{itemize}
\end{prop}
\begin{proof}
  We consider the change of variables $(x, Y)\ni \Rr^2_-\mapsto (x, y)\in \Omega$ where $y=\rho(x, Y)=e^{Y|D|}\eta(x)+Y$ is periodic in $x$. This change of variables is one-to-one and onto since $\p_Y\rho=1+e^{Y|D|}|D|\eta(x)\ge \mez$ for sufficiently small $\eps$. Define the inverse by  
  \bq\label{def:upsilon}
  (x, y)= (x, \rho(x, Y))\quad\text{if and only if }\quad Y=\chi(x, y).
  \eq
  From the relation 
  \[
  y-Y-\frac{1}{2\pi}\hat \eta(0)=\sum_{k\ne 0}e^{|k|Y}e^{ixk}\hat \eta(k)
  \]
   we have 
  \bq\label{decay:rho}
  |y-Y-\frac{1}{2\pi}\hat \eta(0)|\le C_0e^{Y}\| \eta\|_{H^1}
  \eq
  and hence 
    \[
  |y-Y-\frac{1}{2\pi}\hat \eta(0)|\le C\eps e^{y}\| \eta\|_{H^1}, \quad C=C(\eta).
  \]
  In other words,
  \bq\label{decay:chi}
 |\rho(x, Y)-Y-\frac{1}{2\pi}\hat \eta(0)|\le C\eps e^Y,\quad |\chi (x, y)-y+\frac{1}{2\pi}\hat \eta(0)|\le C\eps e^y
  \eq
  and analogously for derivatives.  A direct calculation shows that if $\wt f(x, Y)=f(x, \rho(x, Y))$ then 
\bq\label{div:eq}
\cnx_{x, Y}(\mathcal{A}\na_{x, Y}\wt f)(x, Y)=\p_Y\rho(\Delta_{x, y}f)(x, \rho(x, Y))
\eq
with 
\bq\label{def:matrixA}
\mathcal{A}=
\begin{bmatrix}
\p_Y\rho & -\p_x\rho\\
-\p_x\rho & \frac{1+|\p_x\rho|^2}{\p_Y\rho}
\end{bmatrix}.
\eq
 Making use of \eqref{decay:chi}, we find that 
\[
| \na^m\cnx_{x, Y}(\mathcal{A}\na_{x, Y}Y)|\le C_m\eps e^Y\quad\forall (x, y)\in \Rr^2_-,~\forall m\ge 0.
\]
 Then by Lemma \ref{lemm:elliptic} below, there exists a unique solution $Z_2^*$ to  
\bq\label{sys:Z*}
\begin{cases}
\cnx_{x, Y}(\mathcal{A}\na_{x, Y}Z^*_2)=-\cnx_{x, Y}(\mathcal{A}\na_{x, Y}Y), \quad (x, y)\in \cO=\Tt\times \Rr_-,\\
\wt Z^*_2(x, 0)=0,\quad x\in \T,\\
\|\na_{x, Y} Z_2^*\|_{H^s(\cO)}\le C'_s\eps\quad\forall s\ge 0,\\
\| e^\frac{-y}{2}\p_x Z_2^*\|_{L^\infty(\T; L^1(\Rr_-))}\le C_*\eps.
\end{cases}
\eq
  Define $Z_2(x, y)$ by $Z_2(x, \rho(x, Y))=Y+Z^*_2(x, Y)$; that is,
\bq\label{form:Z2}
 Z_2(x, y)=\chi(x, y)+Z_2^*(x, \chi(x, y)).
\eq
Then, in view of \eqref{div:eq}, $Z_2$ satisfies 
\bq\label{sys:Z2}
\begin{cases}
\Delta_{x, y}Z_2(x, y)=0\quad\forall ~(x, y)\in\Omega,\\
Z_2(x+2\pi, y)=Z_2(x, y)\quad\forall ~(x, y)\in \Omega,\\
 Z_2(x, \eta(x))=0.
\end{cases}
\eq 
 Moreover,  $Z_2$ is even in $x$ because $\eta$ is even.  
 We claim that 
 \[
 (x, y)\mapsto \int_{-\infty}^y\p_x Z_2(x, y')dy'
 \]
  is well defined as a function in $L^\infty(\Omega)$. Indeed, differentiating \eqref{form:Z2} in $x$ gives 
 \[
  \p_xZ_2(x, y)= \p_x\chi(x, y)+ \p_xZ_2^*(x, \chi(x, y))+\p_YZ_2^*(x, \chi(x, y))\p_x \chi(x, y).
  \] 
  Then using the change of variables $(x, Y)=(x, \chi(x, y))$ and the exponential decay of $\p_xZ_2^*$ (the last estimate in \eqref{sys:Z*}),  together with  \eqref{decay:chi}, we obtain the claim.  Now we can define 
 \bq\label{def:Z1}
 Z_1(x, y)=x-\int_{-\infty}^y\p_x Z_2(x, y')dy',
 \eq
 so that $\p_yZ_1=-\p_xZ_2$. Since $\p_Y Z_2^*\in H^\infty(\cO)$, where $\cO=\Tt\times \Rr_-$,  
 we have $\p_YZ_2^*(x, Y)\to 0$ as $Y\to -\infty$.  Hence
 \[
\lim_{y\to -\infty} \p_yZ_2(x, y)=\lim_{y\to -\infty}\Big(\p_y\chi(x, y)+\p_YZ_2^*(x, \chi(x, y))\p_y\chi(x, y)\Big)= 1
 \]
uniformly for $x\in \Rr$. Together with the fact that $Z_2$ is harmonic, this yields
\[
\p_xZ_1(x, y)=1-\int_{-\infty}^y\p_x^2 Z_2(x, y')dy'=1+\int_{-\infty}^y\p_y^2 Z_2(x, y')dy'=\p_yZ_2(x, y).
\]
Thus $Z_1$ and $Z_2$ obey the Cauchy-Riemann equations
  \begin{align}\label{CR1}
  &\p_xZ_1=\p_yZ_2,\quad\p_yZ_1=-\p_xZ_2\quad\text{in}~\Omega.
  \end{align} 
  But $ \|\na_{x, y} (Z_2-y)\|_{L^\infty(\Omega)}\le C\eps$ due to \eqref{sys:Z*} and \eqref{decay:chi}, 
so that 
$\| \na_{x, y} (Z_1-x)\|_{L^\infty(\Omega)}\le C\eps$, proving (vi).  
Moreover, from \eqref{def:Z1} and the fact that $Z_2$ is even in $x$, it follows that  $Z_1$ is odd in $x$ 
and $Z_1(x+2\pi, y)=2\pi +Z_1(x, y)$.       Finally let us prove (ii). 
Owing to (vi), $Z=Z_1+iZ_2$ is one-to-one for sufficiently small $\eps$. 
By the maximum principle, $Z_2(x, y)\le 0$ in $\Omega$ and hence $Z(\Omega)\subset \Rr^2_-$.  
Then, since $Z$ is continuous, it is onto provided that $Z(\{(x, \eta(x)):x\in \Rr\})=\{(x, 0): x\in \Rr\}$.   
This in turn will follow if $Z_1(\{(x, \eta(x)):x\in \Rr\})=\Rr$. 
Indeed, since  $Z_1$ is continuous and  $Z_1(x, \eta(x))\to \pm \infty$ as $x\to \pm \infty$ in view of \eqref{def:Z1}, 
we conclude the proof. 
\end{proof}

\begin{lemm}\label{lemm:elliptic}
 Assume that $F:\mathcal{O}\to \Rr$ satisfies $\langle y \rangle^\sigma F\in L^2(\mathcal{O})$ for some $\sigma>1$, 
 where  $\mathcal{O}=\T\times\Rr_-$ and $\langle y\rangle =\sqrt{1+y^2}$. Recall the matrix $\mathcal{A}$  given by \eqref{def:matrixA}.

1)  There exists a unique variational solution $u$ to the  linear problem
\bq\label{elliptic:S}
\begin{cases}
\cnx_{x, y} (\mathcal{A} \na_{x, y} u)(x, y)=F(x, y),\quad (x, y)\in \mathcal{O},\\ 
u(x, 0)=0,\quad x\in \Tt
\end{cases}
\eq
such that 
\bq\label{variest}
\| \langle y\rangle^{-\sigma} u\|_{L^2(\mathcal{O})}\le \wt C_1\| \na_{x, y} u\|_{L^2(\mathcal{O})}  
\le \wt C_2\| \langle y\rangle^\sigma F\|_{L^2(\mathcal{O})}.
\eq
2) If  $F\in C^\infty(\overline{\mathcal{O}})$  satisfies $|\na_{x, y}^mF(x, y)|  
\le C_m\eps e^y$ in $\mathcal{O}$ for all $m\ge 0$, then 
\begin{align}\label{Sobolev:elliptic}
&\| \na_{x, y} u\|_{H^s(\cO)} \le C'_s\eps\quad\forall s\ge 0 \quad \text{ and } \\ \label{decay:elliptic}  
&\| e^\frac{-y}{2}\p_x u\|_{L^\infty(\T; L^1(\Rr_-))} \le  C_*\eps.
\end{align}
\end{lemm}

\begin{proof} 
We only need to be careful with the behavior as $y\to -\infty$. 
In order to find the variational solution, we need a weighted Poincare inequality. Indeed,  it is easy to see that for any $\sigma>1$, there exists $C>0$ such that 
\bq\label{Poincare}
\int_{\mathcal{O}} \langle y\rangle^{-2\sigma}|u(x, y)|^2dydx\le C\int_{\mathcal{O}}| \p_y u(x, y)|^2dydx
\eq
for all $u(x, y)\in C_0^\infty(\mathcal{O})$. Define $\mathcal{H}^1_0(\mathcal{O})$ to be 
the completion of $C^\infty_0(\mathcal{O})$ under the norm 
\[
\| u\|_{\mathcal{H}^1_0}=\| \langle y\rangle^{-\sigma}u\|_{L^2(\mathcal{O})}+\| \na_{x, y} u\|_{L^2(\mathcal{O})}.
\]
Owing to \eqref{Poincare}, $\mathcal{H}^1_0(\mathcal{O})$ is a Hilbert space with respect to the inner product  
\[
(u, v)_{\mathcal{H}^1_0(\mathcal{O})}:=(\na_{x, y} u, \na _{x, y}v)_{L^2(\mathcal{O})}.
\]
 The Lax-Milgram theorem implies that  the elliptic problem \eqref{elliptic:S} has a unique solution  $u\in \mathcal{H}^1_0(\mathcal{O})$. More precisely, $u$ satisfies 
\bq\label{variform}
\int_{\mathcal{O}} \mathcal{A}\na_{x, y} u\cdot \na_{x, y} \varphi dydx=\int_{\mathcal{O}} F\varphi dydx
\eq
for all $\varphi \in \mathcal{H}^1_0(\mathcal{O})$. Inserting $\varphi=u$ yields the variational estimate \eqref{variest}. 

 Now we prove the decay estimates 2). Assume that   $F\in C^\infty(\overline{\mathcal{O}})$  satisfies 
$|\na_{x, y}^mF(x, y)|\le C_m\eps e^y$ in $\mathcal{O}$ 
for all $m\ge 0$. Then $F\in H^\infty(\mathcal{O})$ and by the standard finite difference technique 
we obtain $\na_{x, y} u\in H^\infty(\mathcal{O})$ together with \eqref{Sobolev:elliptic}. 
It remains to prove the decay \eqref{decay:elliptic}.  Let us rewrite \eqref{elliptic:S}
 as
\bq
\begin{cases}
\Delta_{x, y} u(x, y)=G:=F+\cnx_{x, y}\big((\mathrm{Id}-\mathcal{A})\na_{x, y} u\big),\quad (x, y)\in\cO,\\
u(x, 0)=0,\quad x\in \Tt.
\end{cases}
\eq
 where   $|\na_{x, y}^m G(x, y)|\le  C'_m\eps e^{y}$ for all $m\ge 0$. Denoting by $\hat u(k, y)$  the Fourier transform of $u$ with respect to $x$, and analogously for $\hat G(k, y)$, we have 
 \[
 -k^2 \hat u(k, y)+\p_y^2\hat u(k, y)=\hat G(k, y),\quad \hat u(k, 0)=0\quad \forall k\in \Zz.
 \]
  The unique solution $\hat u$ that guarantees $\na_{x, y} u\in L^2(\cO)$ is given by 
\[
\hat u(k, y)=-\frac{e^{-|k|y}}{2|k|}\int_{-\infty}^ye^{|k|y'}\hat G(k, y')dy'+\frac{e^{|k|y}}{2|k|}\left(\int_{-\infty}^0e^{|k|y'}\hat G(k, y')dy'-\int_y^0e^{-|k|y'}\hat G(k, y')dy'\right)
\]
for all $k\ne 0$ and 
\[
\hat u(0, y)=-\int_y^0\int_{-\infty}^{y'}\hat G(0, y'')dy''dy'.
 \]
  Using $|\hat G(k, y)|\le  C'\eps \frac{e^y}{|k|+1}$ for all $k\in \Zz$, we estimate
\[
|k\hat u(k, y)|\le 
\begin{cases}
 \frac{C'\eps}{2(|k|+1)}\left[\frac{e^y}{|k|+1}+\frac{e^{|k|y}}{|k|+1}+\frac{1}{|k|-1}\big(e^y-e^{|k|y}\big)\right],\quad |k|\ge 2,\\
\frac{C'\eps}{4}\big(e^y-ye^y \big),\quad |k|=1.
\end{cases}
\]
Integrating in $y$, we obtain 
\[
\int_{\Rr_-}|e^\frac{-y}{2}k\hat u(k, y)|dy\le C''\eps |k|^{-2}\quad\forall |k|\ge 1. 
\]
Hence  $\| e^\frac{-y}{2}\p_x u\|_{L^\infty(\T; L^1(\Rr_-))}\le C_*\eps$, thereby proving \eqref{decay:elliptic}.    
In fact, the same decay can be proved for all derivatives of $u$.
\end{proof}

\subsection{ Proof of Proposition \ref{prop:z12}}
Applying Proposition \ref{prop:Riemann} with $\eta(x)=\eta^*(x)=O_\eps(\eps)$ we obtain a Riemann mapping 
$Z_1(x, y)+iZ_2(x,  y)$ from $\{(x, y)\in \Rr^2: y<\eta^*(x)\}$ onto ${\Rr^2_-}$. Let $z_1+iz_2$ be the 
inverse of $Z_1+iZ_2$. The properties (iii), (v) and (vi) in Proposition \ref{prop:Riemann} imply that 
\[
z_1(x+2\pi, y)=2\pi+ z_1(x, y),\quad z_2(x+2\pi, y)=z_2(x, y)\quad\forall(x, y)\in {\Rr^2_-},
\]
 $z_1$ is odd in $x$ and $z_2$ is even in $x$, and $\|\na_{x, y}(z_1-x)\|_{L^\infty(\cO)}+ \|\na_{x, y}(z_2-y)\|_{L^\infty(\cO)}\le C\eps$. An alternative way to state the  even-odd property is 
 $-\overline{z(x+iy)}=z(-(\overline{x+iy}))$, where $z=z_1+iz_2$.   
 Then $\zeta(x)=z_1(x, 0)$ is odd and  $z_2$ satisfies 
\bq\label{sys:z2}
\begin{cases}
\Delta_{x, y}z_2=0\quad\text{in}~{\Rr^2_-},\\
z_2(x+2\pi, y)=z_2(x, y)\quad \forall(x, y)\in {\Rr^2_-},\\
 z_2(x, 0)=\eta^*\circ \zeta(x),\\
 \na_{x, y} (z_2(x, y)-y)\in L^\infty(\cO).
\end{cases}
\eq
It follows that
\bq\label{z2:1}
z_2(x, y)=y+\frac{1}{2\pi}\sum_{k\in \Zz}e^{ikx}e^{y|k|}\widehat{\eta^*\circ \zeta}(k).
\eq
Using the Cauchy-Riemann equations we find that
\[
z_1(x, y)=R+x-\frac{i}{2\pi}\sum_{k\ne 0}e^{ikx}\mathrm{sign}(k)e^{y|k|}\widehat{\eta^*\circ \zeta}(k)
\]
for some constant $R\in \Rr$. Finally, since $z_1$ is odd, we have $R=0$ and hence \eqref{z1} follows.

\subsection{Proof of Lemma \ref{lemm:Riemann}}\label{Appendix:Gzeta}
 For $f\in  H^1(\T_L)$, we first recall from \eqref{def:Gh} and \eqref{elliptic:G} that 
\bq\label{def:Gh2}
G(\eta^*)f=\p_y\tt(x, \eta^*(x))-\p_x\tt(x, \eta^*(x))\p_x\eta^*(x)
\eq
where $\tt(x, y)$ solves the elliptic problem 
\bq\label{elliptic:G2}
\begin{cases}
\Delta_{x, y}\tt=0\quad\text{in}~\Omega,\\
\tt\vert_{y=\eta^*(x)}=f(x),\quad \na_{x, y}\tt\in L^2(\Omega).
\end{cases}
\eq
Let $z(\ul x, \ul y)=z_1+iz_2$ be the Riemann mapping given by Proposition \ref{prop:z12}. Set $\Theta(\ul x, \ul y)=\tt(z_1(\ul x, \ul y), z_2(\ul x, \ul y))$ for $(\ul x, \ul y)\in {\Rr^2_-}$. Since $z$ is holomorphic and $\tt$ is harmonic in $\Omega$, $\Theta$ is harmonic in ${\Rr^2_-}$. Next we find the boundary conditions for $\Theta$. Recall that $z$ maps $\{(\ul x, 0): x\in\Rr\}$ onto $\{(x, \eta^*(x): x\in \Rr)\}$. It follows that  $z_1(\ul x, 0)=\zeta(\ul x)$ and $z_2(\ul x, 0)=\eta^*(\zeta(\ul x))=(\zeta_\sharp \eta^*)(\ul x)$. In addition, $\| \na_{\ul x, \ul y}(z_1-\ul x)\|_{L^\infty({\Rr^2_-})}+\| \na_{\ul x, \ul y}(z_2-\ul y)\|_{L^\infty({\Rr^2_-})}\le C\eps$ by (iv) in Proposition \ref{prop:z12}.  Thus $\Theta$ satisfies
\[
\begin{cases}
\Delta_{\ul x, \ul y}\Theta=0\quad\text{in}~{\Rr^2_-},\\
\Theta(\ul x, 0)=(\zeta_\sharp f)(\ul x),\quad \na_{\ul x, \ul y}\Theta\in L^2(\T_L\times\Rr_-).
\end{cases}
\]
$\Theta$ is given explicitly by
\[
\Theta(\ul x, \ul y)=\frac{1}{L}\sum_{k\in \Zz} e^{ik\frac{2\pi}{L} \ul x}e^{\ul y|k\frac{2\pi}{L}|}\wh{(\zeta_\sharp f)}^L(k).
\]
In particular, 
\[
\p_y\Theta(\ul x, 0)=\frac{1}{L}\sum_{k\in \Zz} e^{ik \frac{2\pi}{L} x}|k\frac{2\pi}{L}|\wh{(\zeta_\sharp f)}^L(k)=|D_L|(\zeta_\sharp f)(\ul x).
\]
On the other hand, by the chain rule and \eqref{sys:z2} and the Cauchy-Riemann equations, we obtain 
\[
\p_y\Theta(\ul x, 0)=\zeta'(\ul x)\Big[\tt_y\big(\zeta(\ul x), \eta^*(\zeta(\ul x)\big)-\tt_x\big(\zeta(\ul x), \eta^*(\zeta(\ul x)\big)\p_x\eta^*(\zeta(\ul x)) \Big]=\zeta'(\ul x)\zeta_\sharp\big( G(\eta^*)f\big)(\ul x).
\]
Combining both expressions for $\p_y\Theta(\ul x, 0)$ yields
\[
\zeta_\sharp\big( G(\eta^*)f\big)(\ul x)=\frac{1}{\zeta'(\ul x)}|D_L|(\zeta_\sharp f)(\ul x)=\frac{1}{\zeta'(\ul x)}\p_{\ul x} \mathcal{H}_L(\zeta_\sharp f)(\ul x),
\]
where $\mathcal{H}_L$ denotes the Hilbert transform, $\wh{\mathcal{H}_Lu}^L(k)=-i\mathrm{sign}(k)\wh{u}^L(\xi)$.  Finally, in view of the identity $\zeta_\sharp^{-1}(\frac{1}{\zeta'}\p_{\ul x}g)=\p_x(\zeta_\sharp^{-1}g)$ with $g=\mathcal{H}_L(\zeta_\sharp f)$, we arrive at the claimed identity  $G(\eta^*)f=\p_x\big(\zeta^{-1}_\sharp \mathcal{H}_L(\zeta_\sharp f)\big)$.
\section{Proof of Lemma \ref{lemm:expandpq}}\label{Appendix:expandpq}
An application of the shape-derivative formula \eqref{shapederi} yields 
\bq  \label{Getapsi}
\begin{aligned}
G(\eta^*)\psi^*&=G(0)\psi^*-G(0)(B^*\eta^*)-\p_x(V^*\eta^*)+O_\eps(\eps^3)\\
&=|D|\psi^*-|D|(B^*\eta^*)-\p_x(V^*\eta^*)+O_\eps(\eps^3),
\end{aligned}
\eq
where in view of \eqref{expand:star} and \eqref{BV},
\[
\begin{aligned}
&B^*=G(\eta^*)\psi^*+O(\eps^2)=|D|\psi^*+O(\eps^2)=\eps \sin x+O_\eps(\eps^2),\\
&V^*=\p_x\psi^*+O(\eps^2)=\eps \cos x+O_\eps(\eps^2).
\end{aligned}
\]
The remainder in \eqref{Getapsi} is $O(\eps^3)$ because both $\eta$ and $\psi$ are $O_\eps(\eps)$. Next we find the $\eps^2$ terms in $B^*$ and $V^*$ from \eqref{BV}, \eqref{Getapsi} and \eqref{expand:star}, obtaining 
\bq\label{B:app}
\begin{aligned}
B^*&=G(\eta^*)\psi^*+\p_x\psi^*\p_x\eta^*+O_\eps(\eps^3)\\
&=|D|\psi^*-|D|(B(0, \psi^*)\eta^*)-\p_x(V(0, \psi^*)\eta^*)+\p_x\psi^*\p_x\eta^*+O_\eps(\eps^3)\\
&=|D|\psi^*-|D|((|D|\psi^*)\eta^*)-\p_x((|D|\psi^*)\eta^*)+\p_x\psi^*\p_x\eta^*+O_\eps(\eps^3)\\
&=|D|(\eps \sin x+\mez\eps^2\sin(2x))-|D|\big((\eps \sin x)(\eps \cos x)\big)-\p_x\big((\eps \cos x)(\eps \cos x)\big)\\
&\qquad-(\eps \cos x)(\eps \sin x)+O_\eps(\eps^3)\\
&=(\eps \sin x+\eps^2\sin(2x))-\eps^2\big\{\sin(2x) -\sin(2x)+\mez\sin(2x)\big\}+O_\eps(\eps^3)\\
&=\eps \sin x+\mez \eps^2 \sin(2x)+O_\eps(\eps^3)
\end{aligned}
\eq
and
\bq
\begin{aligned}\label{V:app}
V^*&=\p_x\psi^*-B^*\p_x\eta^*=\eps \cos x+\eps^2\cos (2x)+(\eps \sin x)(\eps \sin x)+O_\eps(\eps^3)\\
&=\eps \cos x+\mez\eps^2(1+\cos(2x))+O_\eps(\eps^3).
\end{aligned}
\eq
Formula  \eqref{z1} gives
\[
\zeta(x)=z_1(x, 0)=x-\frac{i}{2\pi}\sum_{k\ne 0}e^{ikx}\mathrm{sign}(k)\widehat{\eta^*\circ \zeta}(k),
\]
where $\zeta=x+\eps \zeta^1+\eps^2\zeta^2+O_\eps(\eps^3)$, $\eta^*=\eps\eta^1+\eps^2\eta^2+O_\eps(\eps^3)$ and 
\begin{align*}
\eta^*\circ \zeta(x)&=\eps\eta^1(x)+\eps^2\{\p_x\eta^1(x)\zeta^1(x)+\eta^2(x)\}+O_\eps(\eps^3)\\
&=\eps\cos x+\eps^2\{-\zeta^1(x)\sin x+\mez \cos (2x)\}+O_\eps(\eps^3).
\end{align*}
Matching the orders of $\eps$ we  find that 
\[
\zeta^1(x)=-\frac{i}{2\pi}\sum_{k\ne 0}e^{ikx}\mathrm{sign}(k)\widehat{\cos}(k)=-i\sign(D)\cos(x)=\sin x
\]
and, with $f(x)=-\zeta^1(x)\sin x+\mez \cos (2x)=-\sin^2x+\mez\cos(2x)$,
\[
\zeta^2(x)=-\frac{i}{2\pi}\sum_{k\ne 0}e^{ikx}\mathrm{sign}(k)\widehat{f}(k)=-i\sign(D)f(x)=\sin(2x).
\]
Thus, we obtain  $\zeta(x)=x+\eps\sin x+\eps^2\sin(2x)+O(\eps^3)$, which  finishes the proof of \eqref{zeta1}. 

Next Taylor-expanding $\zeta_\sharp V^*(x)=V^*(\zeta(x))$ using \eqref{zeta1} and \eqref{V:app} gives
\[
\zeta_\sharp {V^*}(x)=\eps \cos x+\eps^2\cos (2x)+O_\eps(\eps^3).
\]
Then combined with the expansion $\zeta'(x)=1+ \eps\cos x+2\eps^2 \cos(2x)+O_\eps(\eps^3)$, this implies 
\[
\begin{aligned}
p(x)&=\frac{c^*-\zeta_\sharp {V^*}}{\zeta'}\\
&=\frac{1+\mez \eps^2- \eps \cos x-\eps^2\cos(2x)}{1+\eps \cos x+2\eps^2 \cos(2x)}+O_\eps(\eps^3)\\
&=\left[1- \eps \cos x+\eps^2\big(\mez-\cos(2x)\big)\right]\left[1-\eps \cos x+\eps^2\big(\mez-\tdm \cos(2x)\big)\right]+O_\eps(\eps^3)\\
&=1 -2\eps\cos x+\eps^2\big(\tdm-2\cos(2x)\big)+O_\eps(\eps^3).
\end{aligned}
\]
Similarly, we have $\zeta_\sharp B^*(x)=\eps\sin x+\eps ^2\sin(2x)+O_\eps(\eps^3)$ and 
\[
q(x)=-p(x)\p_x(\zeta_\sharp B^*)(x)=-\eps\cos x+\eps^2(1-\cos(2x))+O_\eps(\eps^3).
\]
 Finally,  we expand
\begin{align*}
\frac{1+q(x)}{\zeta'}-1&=\frac{1-\zeta'+q}{\zeta'}\\
&=\Big[-2\eps \cos x+\eps^2\big(1-3\cos(2x)\big)\Big]\Big[1-\eps \cos x\Big]+O_\eps(\eps^3)\\
&=-2\eps \cos x+2\eps^2(1-\cos(2x))+O_\eps(\eps^3)\\
\end{align*}
which completes the proof.  
\section{Higher-order expansions} 
At a certain point in our investigation we expected that higher-order expansions would be necessary.  
We share these expansions with the reader in the expectation that they might well be useful in 
future computational and theoretical work.  
 \begin{align}
&\wt U_2=
 \begin{bmatrix}
 \mathrm{odd}\\  \mathrm{even}
 \end{bmatrix}=
\eps\begin{bmatrix}-S\\  C
\end{bmatrix} 
+\eps^2\begin{bmatrix}
-2S_2\\
C_2
\end{bmatrix}
+\eps^3\begin{bmatrix}
S-\frac{9}{2}S_3\\
-\mez C+\tdm C_3
\end{bmatrix}
+O(\eps^4),\\ 
&  U_3=
 \begin{bmatrix}
 \mathrm{ even}\\  \mathrm{odd}
 \end{bmatrix}=\begin{bmatrix} C\\ S
\end{bmatrix}
+\eps
\begin{bmatrix}
2C_2\\ S_2
\end{bmatrix}
+\eps^2\begin{bmatrix}
\frac{9}{2}C_3\\
-\mez S+\tdm S_3
\end{bmatrix}+O(\eps^3),\\
  & U_4=
 \begin{bmatrix}
  \mathrm{even}\\  \mathrm{odd}
 \end{bmatrix}=\begin{bmatrix}1\\ 0 \end{bmatrix}+ 
\eps
\begin{bmatrix}
C\\
-S
\end{bmatrix}
+\eps^2
\begin{bmatrix}
2C_2\\ -S_2
\end{bmatrix}
+O(\eps^3).
 \end{align}
\bq
\begin{aligned}
&{\bf A}_{11}=i\mu +\tdm i\mu\eps^2+\mu O_\eps(\eps^3),\quad {\bf A}_{12}=-i\mu\eps +\mu O_\eps(\eps^3),\\
& {\bf A}_{13}=\mu O_\eps(\eps^3),\quad {\bf A}_{14}=\mu-\mu\eps^2 +\mu O_\eps(\eps^3),\\
& {\bf A}_{21}=-i\mu\eps+\mu O_\eps(\eps^3),\quad {\bf A}_{22}=\mez i\mu-\frac{5}{4}i\mu\eps^2+\mu O_\eps(\eps^3),\\
&{\bf A}_{23}= {\bf A}_{24}=0,\\
&{\bf A}_{31}=O_\eps(\eps^5),\quad {\bf A}_{32}=-\eps^2,\\
&{\bf A}_{33}=\mez i\mu -\frac{5}{4}i\mu\eps^2+\mu O_\eps(\eps^3),\quad {\bf A}_{34}=-\tdm i\mu\eps+\mu O_\eps(\eps^3),\\
&{\bf A}_{41}=-1,\quad {\bf A}_{42}=0,\\
&{\bf A}_{43}=-\mez i\mu\eps +\mu O_\eps(\eps^3),\quad {\bf A}_{44}=i\mu +\mu O_\eps(\eps^3).
\end{aligned}
\eq
\\ \\
\noindent{\bf{Acknowledgment.}} 
The work of HQN was partially supported by NSF grant DMS-1907776. We thanks B. Deconinck and R. Creedon for useful discussions. 
\vspace{1mm}

\end{document}